\documentclass[11pt]{article}

\usepackage[dvips]{epsfig} 
\usepackage{amssymb,amsmath,amsthm,mathrsfs} 
\usepackage{graphicx}
\usepackage{lineno}
\usepackage{multirow}

\usepackage[authoryear,sort&compress]{natbib}
\usepackage{url}
\usepackage{enumerate}
\usepackage{colordvi,color,pspicture}
\usepackage{psfrag}
\usepackage{pgf,tikz}
\usepackage{mathrsfs}
\usetikzlibrary{arrows}
\usetikzlibrary{decorations.markings}

\newcommand{\E}{\mathbf{E}}
\newcommand{\Var}{\mathbf{Var}}
\newcommand{\Cov}{\mathbf{Cov}}

\newtheorem{thm}{Theorem}[section]
\newtheorem{cor}[thm]{Corollary}
\newtheorem{lemma}[thm]{Lemma}

\newtheorem{prop}[thm]{Proposition}

\newtheorem{remark}[thm]{Remark}

\theoremstyle{definition}
\newtheorem{definition}[thm]{Definition}

\begin{document}
\title{On the Distribution of the Number of Copies of Weakly Connected Digraphs in Random $k$NN Digraphs}
\author{Selim Bahad{\i}r \& Elvan Ceyhan
\\ \small Department of Mathematics \\
\small Ko\c{c} University, \.{I}stanbul, Turkey}
\date{\today}

\maketitle
\pagenumbering{roman} \setcounter{page}{1}

\pagenumbering{arabic} \setcounter{page}{1}

\begin{abstract}
\noindent
In a digraph with $n$ vertices, a minuscule construct is a subdigraph with $m<<n$ vertices.
We study the number of copies of a minuscule constructs in $k$ nearest neighbor ($k$NN) digraph of the data from a random point process in $\mathbb{R}^d$.
Based on the asymptotic theory for functionals of point sets under homogeneous Poisson process and binomial point process, we provide a general result for the asymptotic behavior of the number of minuscule constructs and
as corollaries,
we obtain asymptotic results for the number of vertices with fixed indegree,
the number of shared $k$NN pairs
and the number of reflexive $k$NN's in a $k$NN digraph.
\end{abstract}

\noindent
{\it Keywords:} asymptotic normality; binomial process; central limit theorem; homogeneous point process; indegree; law of large numbers; reflexivity\\

\section{Introduction}
\label{sec:intro}

Random graph models such as Erd\H{o}s-R\'{e}nyi graphs, random geometric graphs and nearest neighbor type graphs
are used in various fields.
The most frequently studied is the one proposed in \cite{gilbert:1959}, denoted $\mathbf{G}(n,p)$,
in which each possible edge between $n$ vertices occurs independently with probability $0<p<1$.
However, in the literature,
$\mathbf{G}(n,p)$ is usually called Erd\H{o}s-R\'{e}nyi model as they developed the theory (\cite{erdos:1960}).
Another commonly considered model is random geometric graphs which are constructed by
randomly placing fixed number of vertices in some metric space (according to a specified probability distribution)
and connecting two vertices by an edge if and only if their distance is smaller than a certain neighborhood radius.
For more information about random geometric graphs, see \cite{penrose:2003}.
The number of copies of a fixed graph in random graphs are analyzed by many authors (e.g.,
\cite{nowicki:1988}, \cite{Rucinski:1988} and \cite{janson:2004} for $\mathbf{G}(n,p)$ and
\cite{najim:2003}, \cite{yu:2009} and \cite{shang:2010} for random geometric graphs).
Throughout this paper,
we consider random $k$NN digraphs and study asymptotic distribution of the number of minuscule constructs
for $k$NN digraphs based on data from a binomial point process or homogeneous Poisson process (HPP).

Let $k,d\geq 1$ be fixed integers,
$n\geq 2$ be an integer and $V$ be a finite set of points in $\mathbb{R}^d$.
For any $v\in V$, let $kNN(v)$ denote the set of $k$ closest points to $v$ among the points in $V\backslash \{v\}$ with respect to a given metric, and whenever $u\in kNN(v)$ we call $u$ as a $k$NN of $v$.
Throughout this article, we consider the usual Euclidean metric and denote the distance between the points $x$ and $y$ in $\mathbb{R}^d$ as $\|x-y\|$.
Obviously, $kNN(v)$ may not be well defined if there exist points $u,w\in V$ such that $\|u-v\|=\|w-v\|$.
However, such a tie occurs with probability zero for the random point sets which are obtained by HPP or binomial point process,
and hence we may always assume that pairwise distances are distinct and $kNN(v)$ is well defined for the point sets under consideration.
The $k$NN directed graph (or digraph) on the point set $V$, denoted as $kNND(V)$, is obtained by including an arc (i.e., directed edge) with tail $v$ and head $u$ whenever $u$ is one of the $k$NN's of $v$.
In other words, $kNND(V)$ is actually the digraph with vertex set $V$ and arc set $A=\{(u,v): u,v\in V , v\in kNN(u) \}$.

In a digraph, \emph{indegree} (resp. \emph{outdegree}) of a vertex $v$ is the number of arcs with head (resp. tail) $v$
and denoted as $d_{in}(v)$ (resp. $d_{out}(v)$).
Notice that the outdegree of a vertex in $kNND(V)$ is always $k$ as long as $V$ has at least $k+1$ vertices.
For $j\geq 0$, let $Q^{(k)}_j(V)$ denote the number of vertices of $kNND(V)$ with indegree $j$, that is, the number of points which are $k$NN of exactly $j$ points in $V$.
The problem of finding the probability that a random point is the NN of precisely $j$ other points is studied by many authors such as \cite{clark:1955}, \cite{roberts:1969}, \cite{newmanRT:1983} and \cite{henze:1987}.
The quantities $Q^{(1)}_j$'s are used in tests for spatial symmetry (see, \cite{ceyhan:2014}).
Also in \cite{enns:1999}, $Q^{(1)}_0, Q^{(1)}_1$ and $Q^{(1)}_2$ correspond to the number of lonely, normal and popular individuals in a population, respectively.

A triplet $(\{u,w\},v)$ with $u,v,w\in V$ is called \emph{shared} $k$NN's whenever $v\in kNN(u)$ and $v\in kNN(w)$; i.e., $v$ is a $k$NN to both $u$ and $w$ in $V$,
and the number of shared $k$NN's in $V$ is denoted as $Q^{(k)}(V)$.
In other words, $Q^{(k)}(V)$ counts pair of arcs sharing their heads in $kNND(V)$.
The quantity $Q^{(k)}$ can be expressed in terms of $Q^{(k)}_j$'s.
By a simple double counting argument, one can easily see that
\begin{align*}
Q^{(k)}(V)=\sum_{v\in V} {d_{in}(v) \choose 2}= \sum_{j\geq 0} \frac{j(j-1)}{2}Q^{(k)}_j(V)
\end{align*}
for any point set $V$.

An ordered pair of vertices $\{u,v\}$ is called a \emph{reflexive} $k$NN pair whenever $u\in kNN(v)$ and $v\in kNN(u)$,
that is, $u$ and $v$ are $k$NN of each other (\cite{cox:1981}).
Other authors have called these pairs as isolated nearest neighbors (\cite{pickard:1982}) or mutual nearest neighbors (\cite{schiling:1986}).
In graph theory, reflexive pairs are also referred to as symmetric arcs (\cite{chartrand:1996}).
We denote the number of reflexive pairs in $kNND(V)$ as $R^{(k)}(V)$.
The quantity $R^{(1)}$ could be of interest for inferential purposes as well,
since it is a measure of mutual (symmetric) spatial dependence between points,
which might indicate a special and/or stronger form of clustering of data points.
For instance, a simple test based on the proportion of the number of
reflexive pairs to the sample size was presented by \cite{dacey:1960}
to interpret the degree of regularity or clustering of the locations of towns alongside a river.

Numbers of reflexive and shared $k$NN pairs are of importance in various fields.
For example,
in spatial data analysis,
the distributions of the tests based on nearest neighbor contingency tables depend on these two quantities
(\cite{cuzick:1990}, \cite{dixon:1994} and \cite{ceyhan:cell2009}).
Moreover, neighbor sharing type quantities such as $Q^{(1)}$ are also of interest for the problem of estimating the intrinsic dimension of a data set (see, \cite{brito:2013}).
For a set of ten points, 1NN and 2NN digraphs are presented in Figure \ref{fig:2NNofV} together with the corresponding $R^{(k)}, Q^{(k)}, Q^{(k)}_j$ values.
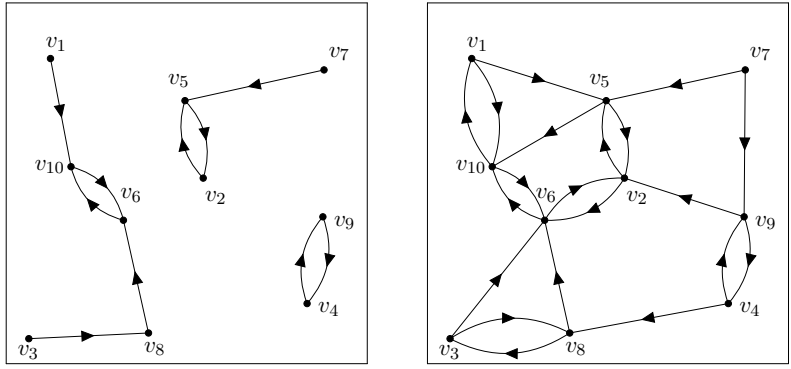
\begin{figure}
\center
\scalebox{.8}{
\begin{tikzpicture}[line cap=round,line join=round,>=triangle 45,x=1.0cm,y=1.0cm]
\begin{scope}
\clip(0.9179294542160873,0.9302415733519224) rectangle (7.099192962929973,7.111505082065798);
\draw (1.,7.)-- (7.,7.);
\draw (7.,7.)-- (7.,1.);
\draw (7.,1.)-- (1.,1.);
\draw (1.,1.)-- (1.,7.);
\tikzset{->-/.style={decoration={
  markings,
  mark=at position #1 with {\arrow{>}}},postaction={decorate}}}

\draw [->-=.55] (1.735784653837004,6.074981229072559) -- (2.0767011370474076,4.2756997899065405);
\draw [->-=.55] (3.9706815993274276,5.374208458028952) to [bend left] (4.273718473292231,4.086301743678539);
\draw [->-=.55] (6.281337763309051,5.885583182844558) -- (3.9706815993274276,5.374208458028952);
\draw [->-=.55] (2.0767011370474076,4.2756997899065405) to [bend left] (2.9479321496962165,3.385528972634931);
\draw [->-=.55] (4.273718473292231,4.086301743678539) to [bend left] (3.9706815993274276,5.374208458028952);
\draw [->-=.55] (2.9479321496962165,3.385528972634931) to [bend left] (2.0767011370474076,4.2756997899065405);
\draw [->-=.55] (6.262397958686251,3.442348386503332) to [bend left] (6.,2.);
\draw [->-=.55] (1.3759283660038002,1.4157892918637107) -- (3.364607851397821,1.5104883149777117);
\draw [->-=.55] (3.364607851397821,1.5104883149777117) -- (2.9479321496962165,3.385528972634931);
\draw [->-=.55] (6.,2.) to [bend left] (6.262397958686251,3.442348386503332);
\draw (1.5205596013051847,6.563836972799612) node[anchor=north west] {$v_1$};
\draw (1.0556734878364524,1.442871720441019) node[anchor=north west] {$v_3$};
\draw (4.172132248497213,4.012264325607587) node[anchor=north west] {$v_2$};
\draw (6.048894706574688,2.152719871732661) node[anchor=north west] {$v_4$};
\draw (3.5566441812564834,5.957898800495242) node[anchor=north west] {$v_5$};
\draw (2.7646919122935623,4.048354346620315) node[anchor=north west] {$v_6$};
\draw (6.224382773815419,6.407220922369064) node[anchor=north west] {$v_7$};
\draw (3.157924013154658,1.5000897246435645) node[anchor=north west] {$v_8$};
\draw (6.313254790625601,3.5990322247464928) node[anchor=north west] {$v_9$};
\draw (1.3377776055077303,4.528804451683956) node[anchor=north west] {$v_{10}$};
\begin{scriptsize}
\draw [fill=black] (2.0767011370474076,4.2756997899065405) circle (1.5pt);
\draw [fill=black] (3.9706815993274276,5.374208458028952) circle (1.5pt);
\draw [fill=black] (6.281337763309051,5.885583182844558) circle (1.5pt);
\draw [fill=black] (6.262397958686251,3.442348386503332) circle (1.5pt);
\draw [fill=black] (4.273718473292231,4.086301743678539) circle (1.5pt);
\draw [fill=black] (1.735784653837004,6.074981229072559) circle (1.5pt);
\draw [fill=black] (6.,2.) circle (1.5pt);
\draw [fill=black] (2.9479321496962165,3.385528972634931) circle (1.5pt);
\draw [fill=black] (3.364607851397821,1.5104883149777117) circle (1.5pt);
\draw [fill=black] (1.3759283660038002,1.4157892918637107) circle (1.5pt);
\end{scriptsize}
\end{scope}

\begin{scope}[shift={(7,0)}]
\clip(0.9179294542160873,0.9302415733519224) rectangle (7.099192962929973,7.111505082065798);
\draw (1.,7.)-- (7.,7.);
\draw (7.,7.)-- (7.,1.);
\draw (7.,1.)-- (1.,1.);
\draw (1.,1.)-- (1.,7.);
\tikzset{->-/.style={decoration={
  markings,
  mark=at position #1 with {\arrow{>}}},postaction={decorate}}}

\draw [->-=.55] (1.735784653837004,6.074981229072559) to [bend left] (2.0767011370474076,4.2756997899065405);
\draw [->-=.55] (3.9706815993274276,5.374208458028952) to [bend left] (4.273718473292231,4.086301743678539);
\draw [->-=.55] (6.281337763309051,5.885583182844558) -- (3.9706815993274276,5.374208458028952);
\draw [->-=.55] (2.0767011370474076,4.2756997899065405) to [bend left] (2.9479321496962165,3.385528972634931);
\draw [->-=.55] (4.273718473292231,4.086301743678539) to [bend left] (3.9706815993274276,5.374208458028952);
\draw [->-=.55] (2.9479321496962165,3.385528972634931) to [bend left] (2.0767011370474076,4.2756997899065405);
\draw [->-=.55] (6.262397958686251,3.442348386503332) to [bend left] (6.,2.);
\draw [->-=.55] (1.3759283660038002,1.4157892918637107) to [bend left] (3.364607851397821,1.5104883149777117);
\draw [->-=.55] (3.364607851397821,1.5104883149777117) -- (2.9479321496962165,3.385528972634931);
\draw [->-=.55] (6.,2.) to [bend left] (6.262397958686251,3.442348386503332);
\draw [->-=.55] (1.735784653837004,6.074981229072559) -- (3.9706815993274276,5.374208458028952);
\draw [->-=.55] (3.9706815993274276,5.374208458028952) -- (2.0767011370474076,4.2756997899065405);
\draw [->-=.55] (6.281337763309051,5.885583182844558) -- (6.262397958686251,3.442348386503332);
\draw [->-=.55] (2.0767011370474076,4.2756997899065405) to [bend left] (1.735784653837004,6.074981229072559);
\draw [->-=.55] (4.273718473292231,4.086301743678539) to [bend left] (2.9479321496962165,3.385528972634931);
\draw [->-=.55] (2.9479321496962165,3.385528972634931) to [bend left] (4.273718473292231,4.086301743678539);
\draw [->-=.55] (6.262397958686251,3.442348386503332) -- (4.273718473292231,4.086301743678539);
\draw [->-=.55] (1.3759283660038002,1.4157892918637107) -- (2.9479321496962165,3.385528972634931);
\draw [->-=.55] (3.364607851397821,1.5104883149777117) to [bend left] (1.3759283660038002,1.4157892918637107);
\draw [->-=.55] (6.,2.) -- (3.364607851397821,1.5104883149777117);
\draw (1.5205596013051847,6.563836972799612) node[anchor=north west] {$v_1$};
\draw (1.0556734878364524,1.442871720441019) node[anchor=north west] {$v_3$};
\draw (4.172132248497213,4.012264325607587) node[anchor=north west] {$v_2$};
\draw (6.048894706574688,2.152719871732661) node[anchor=north west] {$v_4$};
\draw (3.5566441812564834,5.957898800495242) node[anchor=north west] {$v_5$};
\draw (2.6646919122935623,4.048354346620315) node[anchor=north west] {$v_6$};
\draw (6.224382773815419,6.407220922369064) node[anchor=north west] {$v_7$};
\draw (3.157924013154658,1.5000897246435645) node[anchor=north west] {$v_8$};
\draw (6.313254790625601,3.5990322247464928) node[anchor=north west] {$v_9$};
\draw (1.3377776055077303,4.528804451683956) node[anchor=north west] {$v_{10}$};
\begin{scriptsize}
\draw [fill=black] (2.0767011370474076,4.2756997899065405) circle (1.5pt);
\draw [fill=black] (3.9706815993274276,5.374208458028952) circle (1.5pt);
\draw [fill=black] (6.281337763309051,5.885583182844558) circle (1.5pt);
\draw [fill=black] (6.262397958686251,3.442348386503332) circle (1.5pt);
\draw [fill=black] (4.273718473292231,4.086301743678539) circle (1.5pt);
\draw [fill=black] (1.735784653837004,6.074981229072559) circle (1.5pt);
\draw [fill=black] (6.,2.) circle (1.5pt);
\draw [fill=black] (2.9479321496962165,3.385528972634931) circle (1.5pt);
\draw [fill=black] (3.364607851397821,1.5104883149777117) circle (1.5pt);
\draw [fill=black] (1.3759283660038002,1.4157892918637107) circle (1.5pt);
\end{scriptsize}
\end{scope}
\end{tikzpicture}}
\caption{
On the left, 1NN digraph of $V=\{v_1,\dots , v_{10} \}$.
Notice that there are three reflexive pairs, namely $\{v_6,v_{10}\}$, $ \{v_2,v_5\}$, $\{v_4,v_9\})$,
and hence $R^{(1)}(V)=3$.
Also, note that indegrees of $v_1, \dots ,v_{10}$ are $0,1,0,1,2,2,0,1,1,2$, respectively, and thus,
$Q^{(1)}(V)=3, Q^{(1)}_0(V)=3, Q^{(1)}_1(V)=4, Q^{(1)}_2(V)=3$ and $Q^{(1)}_j(V)=0$ for every $j\geq 3$.
On the right, 2NN digraph of $V$.
Notice that $R^{(2)}(V)=6, Q^{(2)}(V)=17, Q^{(2)}_0(V)=1, Q^{(2)}_1(V)=3, Q^{(2)}_2(V)=2,  Q^{(2)}_3(V)=3, Q^{(2)}_4(V)=1 $
and $Q^{(2)}_j(V)=0$ for every $j\geq 5$.
} \label{fig:2NNofV}
\end{figure}

For any set $A\subset \mathbb{R}^d$, let $m(A)$ denote the Lebesgue measure of the set $A$ and $\partial A$ denote the boundary of $A$.
Let $B_0$ be a fixed bounded Borel set in $\mathbb{R}^d$ with $m(B_0)>0$ and $m(\partial B_0)=0$.
In this paper, we investigate the asymptotic distributions of $R^{(k)}$, $Q^{(k)}_j$ and $Q^{(k)}$
for HPP of high intensity on $B_0$, and large independent samples of non-random size from the uniform distribution on $B_0$
(i.e., sample from a binomial process on $B_0$).
Using the results of \cite{penrose:2001,penrose:2002},
we provide LLN and CLT results on the number of copies of weakly connected digraphs in $kNND$ under HPP and binomial process in general.
As corollaries of these results, we obtain asymptotic results for $R^{(k)}$, $Q^{(k)}_j$ and $Q^{(k)}$.
A crucial condition on the minuscule construct or the subdigraph is connectedness;
and another is number of vertices of the subdigraph should be fixed.

Section \ref{sec:prelim} presents the main result and its proof is given in Section \ref{sec:proofmain}.
We study the asymptotic behavior of the number of vertices with a given indegree in Section \ref{sec:jindegree}.
In Section \ref{sec:shared} and \ref{sec:reflexive},
we provide asymptotic results for the number of shared pairs and the number of reflexive pairs, respectively.
Under a special setting, we study pairwise dependence of these quantities in Section \ref{sec:depRQQj}.
Discussion and conclusions are provided in Section \ref{sec:tr2discconc}.

\section{Preliminaries}
\label{sec:prelim}
A \emph{directed graph} (or simply \emph{digraph}) $D$ consists of a non-empty set
$V (D)$ of elements called \emph{vertices}
and a set $A(D)$ of ordered pairs of distinct vertices called \emph{arcs} (or \emph{directed edges}).
We call $V (D)$ the vertex set and $A(D)$ the arc set of $D$.
A \emph{graph} $G$ is a non-empty set $V(G)$ of elements called \emph{vertices} together with a set $E(G)$ of
unordered pairs of vertices of $G$ called \emph{edges}.
An edge $\{u,v\}$ is denoted by $uv$ for convenience in the text.
A graph or a digraph is \emph{finite} if its vertex set is finite.
A $u-v$ \emph{path} in a graph $G$ is a sequence of pairwise distinct vertices $u=u_1,u_2,\dots ,u_m=v$ such that $u_i u_{i+1}$ is an edge in $G$ for each $1\leq i \leq m-1$, and
the \emph{length} of the path is the number of edges in the path.
A graph $G$ is called \emph{connected} if there exists a $u-v$ path for every pair of vertices $u$ and $v$ in $G$.
The \emph{distance} between the vertices $u$ and $v$ of a connected graph $G$ is the length of a shortest $u-v$ path.
The \emph{underlying graph} of a digraph $D$ is the graph obtained by replacing each arc with an (undirected) edge,
disallowing multiple edges between two vertices.
A digraph is called \emph{weakly connected} if its underlying graph is connected.

A digraph $D_1$ is a \emph{subdigraph} of a digraph $D_2$ if $V(D_1)\subseteq V(D_2)$ and $A(D_1) \subseteq A(D_2)$.
A digraph $D_1$ is \emph{isomorphic} to a digraph $D_2$ (or $D_1$ and $D_2$ are \emph{isomorphic})
if there exists a bijection $f:V(D_1)\rightarrow V(D_2)$ such that
$(u,v)\in A(D_1)$ if and only if $(f(u),f(v))\in A(D_2)$.

Let $D$ be a fixed weakly connected digraph.
For any finite point set $V$ in $\mathbb{R}^d$,
let $H_{D}(V)$ denote the number of subdigraphs of $kNND(V)$ isomorphic to $D$.
In our setting,
a weakly connected subdigraph with fixed number of vertices is referred to a minuscule construct,
and we are interested in the random variable $H_{D}(V)$
when $V$ consists of random points from HPP or binomial process.
For example, if $D$ is the digraph with $V(D)=\{1,2\}$ and $A(D)=\{(1,2), (2,1) \}$, then
we have $R^{(k)}(V)= H_{D}(V)$.
Similarly, we have $Q^{(k)}(V)= H_{D}(V)$ whenever $D$ is the digraph with $V(D)=\{1,2,3\}$ and $A(D)=\{ (1,2),(3,2) \}$.

Let $(X_n)_{n\geq 1}$ be a sequence of random variables and $x$ be a constant.
If $\lim_{n \rightarrow \infty} \E(X_n)=x$, then we write $X_n \xrightarrow{c.m.} x$ as $n\rightarrow \infty$ (convergence of means).
If $\sum_n P(|X_n-x|> \epsilon )< \infty$ for every $\epsilon >0$,
then we say $X_n$ \emph{converges completely} to $x$ and denote $X_n\xrightarrow{c.c.} x$ as $n\rightarrow \infty$.
We use the notation $\xrightarrow{c.m.c.c.}$, if both types of convergence hold.
Notice that complete convergence implies almost sure convergence but not vice versa.

Let $\mathcal{U}_n=\{U_1,\dots , U_n \}$, where $U_1,\dots , U_n$ are i.i.d. uniform random variables on $B_0$.
Also, let $\mathcal{P}_n$ be the HPP of intensity $n/m(B_0)$ on $B_0$.

\begin{thm}[Main Theorem]
\label{thm:mainllnclt}
Let $m$ be a given positive integer, $D_1,\dots, D_m$ be finite weakly connected digraphs,
$a_1,\dots, a_m$ be real numbers and $H(V)=a_1H_{D_1}(V)+\cdots +a_m H_{D_m}(V)$ for every finite $V\subset \mathbb{R}^d$.
Then there exist constants $\xi, \tau^2,\sigma^2$ with $0 \leq \tau^2 \leq \sigma^2$such that
as $n\rightarrow \infty$,
\begin{gather*}
n^{-1} H(\mathcal{U}_n) \xrightarrow{c.m.c.c.} \xi , \\
  n^{-1} \Var (H(\mathcal{U}_{n})) \rightarrow \tau^2,\\
n^{-1/2} ( H(\mathcal{U}_{n})-\E(H(\mathcal{U}_{n})) ) \xrightarrow{\mathcal{L}} \mathcal{N}(0,\tau^2),\\
n^{-1} H(\mathcal{P}_n) \xrightarrow{c.m.c.c.} \xi,\\
  n^{-1} \Var (H(\mathcal{P}_{n})) \rightarrow \sigma^2,\\
n^{-1/2} ( H(\mathcal{P}_{n})-\E(H(\mathcal{P}_{n})) ) \xrightarrow{\mathcal{L}} \mathcal{N}(0,\sigma^2),
\end{gather*}
where $\xrightarrow{\mathcal{L}}$ denotes convergence in law and
$\mathcal{N}(a,b^2)$ is the normal distribution with mean $a$ and variance $b^2$.
Moreover, $\xi$, $\tau^2$ and $\sigma^2$ are independent of the choice of $B_0$.
\end{thm}
The proof of Theorem \ref{thm:mainllnclt} is given in Section \ref{sec:proofmain}.

\section{Proof of the Main Theorem}
\label{sec:proofmain}

We first borrow some notation and definitions regarding the CLT and LLN results from \cite{penrose:2001,penrose:2002}.

For any set $A\subset \mathbb{R}^d$ and $y\in  \mathbb{R}^d$, denote by $A+y$ the set $\{ x+y:\ x\in A\}$.
Also, for any $c\in \mathbb{R}$, let $cA$ denote the set $\{cx:x\in A\}$.
For $x\in \mathbb{R}^d$ and $r>0$, let $B_r(x)$ denote the Euclidean open ball centered at $x$ and with radius $r$.
Let ${\rm card} (A)$ denote the cardinality of $A$.
Let $\mathcal{P}$ be HPP of unit intensity on $\mathbb{R}^d$.
Let $B_0$ be a fixed bounded Borel set in $\mathbb{R}^d$ with positive volume and $m(\partial (B_0))=0$.

Let $H$ be a real valued functional defined for all finite subsets of $\mathbb{R}^d$.
$H$ is called \emph{translation-invariant}, if $H(V+y)=H(V)$ for all $V \subset \mathbb{R}^d$ and $y\in \mathbb{R}^d$, and
\emph{scale-invariant}, if $H(cV)=H(V)$ for all $V \subset \mathbb{R}^d$ and $c\neq 0$.
Notice that the functionals we consider only depend on the ordering of the pairwise distances between sample points.
Henceforth, these functionals are translation-invariant and scale-invariant, and thus,
without loss of generality, we may assume that $B_0$ is of unit volume and contains the origin.

We say $H$ is \emph{linearly bounded}, if there is a constant $c_1$ such that $|H(V)|\leq c_1 \cdot \text{card}(V)$ for every finite $V \subset \mathbb{R}^d$.
Remaining conditions on $H$ are defined in terms of the \emph{add one cost},
meaning that the increment in $H$ caused by inserting a point at the origin into a finite point set $V \subset \mathbb{R}^d$, which is formally given by
\begin{align}
\Delta_H (V):= H(V\cup \{0\})-H(V). \nonumber
\end{align}
The functional $H$ has \emph{bounded add one cost}, if there exists a constant $c_2$ such that $|\Delta_H ( V )|\leq  c_2$
for every finite point set $V \subset \mathbb{R}^d$.

The functional $H$ is called \emph{strongly stabilizing}, if there exist a.s. finite random variables $S$
(a \textit{radius of stabilization} of $H$) and $\Delta_H (\infty)$ such that with probability 1, we have
$$
\Delta_H ( (\mathcal{P} \cap B_S(0) )\cup A)= \Delta_H (\infty)
$$
for all finite $A\subset \mathbb{R}^d \backslash B_S(0)$.

Hence, $S$ is a radius  of stabilization, if the add one cost is unaffected by
the changes in the configuration outside the ball $B_S(0)$.
In other words,
$$
\Delta_H ( (\mathcal{P} \cap B_S(0) )\cup A_1)=\Delta_H ( (\mathcal{P} \cap B_S(0) )\cup A_2)
$$
for every finite $A_1,A_2 \subset \mathbb{R}^d \backslash B_S(0)$ and this add one cost is denoted by $\Delta_H (\infty)$.
Notice that if $H$ has a radius of stabilization and bounded add one cost, then it is strongly stabilizing.

\subsection{CLT Results}
One can easily obtain the following proposition by applying Theorem 2.1 in \cite{penrose:2001}.
\begin{prop}
\label{prop:mainclt}
Suppose that $H$ is translation-invariant, scale-invariant, linearly bounded, has a radius of stabilization and
bounded add one cost.
Then there exist constants $\sigma^2$ and $\tau^2$, with $0\leq \tau^2 \leq \sigma^2$, such that as $n\rightarrow \infty$,
$\Var(H(\mathcal{P}_n))/n$ converges to $\sigma^2$, $\Var (H(\mathcal{U}_{n}))/n$ converges to $\tau^2$
and both of $H(\mathcal{P}_n)$ and $H(\mathcal{U}_{n})$ are asymptotically normal.
Also, $\sigma^2$ and $\tau^2$ are independent of the choice of $B_0$.
Moreover, if the distribution of $\Delta_H (\infty)$ is non-degenerate, then $\tau^2>0$ and hence also $\sigma^2>0$.
\end{prop}

\begin{lemma}
\label{lem:linearsum}
If $H_i$ satisfies the conditions of Proposition \ref{prop:mainclt} for each $1\leq i \leq m$, so does $a_1H_1+\cdots +a_m H_m$ for any real numbers $a_1, \dots , a_m$.
\end{lemma}
\begin{proof}
First note that by applying induction on $m$, it suffices to prove the claim for $m=2$ .
Let $H=a_1H_1+a_2 H_2$.
$$
H(V+y)= a_1H_1(V+y)+a_2 H_2(V+y)=a_1H_1(V)+a_2 H_2(V)=H(V)
$$
for all $V \subset \mathbb{R}^d$ and $y\in \mathbb{R}^d$ since $H_1$ and $H_2$ are translation-invariant,
and hence $H$  is translation-invariant as well.

As $H_1$ and $H_2$ are scale-invariant, for all $V \subset \mathbb{R}^d$ and every nonzero $a \in \mathbb{R}$, we have
$$
H(aV)=a_1 H_1(aV)+a_2 H_2(aV)=a_1 H_1(V)+a_2 H_2(V)=H(V)
$$
which implies that $H$ is scale-invariant too.

As both $H_1$ and $H_2$ are linearly bounded and have bounded add one costs,
there exist constant $c_1$ and $c_2$ such that $|H_i(V)|\leq c_i \cdot \text{card}(V)$ and $|\Delta_{H_i} (V)|\leq c_i$ , for $i=1,2$.
So, by triangular inequality, we have
\begin{align}
|H(V)|=|a_1 H_1(V)+a_2 H_2(V)|\leq (|a_1|c_1+|a_2|c_2) \cdot \text{card}(V), \nonumber
\end{align}
which implies that $H$ is linearly bounded.

Let $S_i$ be a radius of stabilization of $H_i$ for $i=1,2$, and $S=\max \{S_1, S_2\}$.
Then, $S$ is a radius of stabilization for both $H_1$ and $H_2$.
Since
\begin{align}
\Delta_H (V)&= H(V\cup \{0\})-H(V)= a_1H_1(V\cup \{0\})+a_2 H_2(V \cup \{0\}) - (a_1 H_1(V)+a_2 H_2(V) ) \nonumber\\
&= a_1\Delta_{H_1} (V)+a_2\Delta_{H_2} (V) \label{eq:deltalinear}
\end{align}
for any finite point set $V\subset \mathbb{R}^d$, we have $\Delta_H (\infty)=a_1\Delta_{H_1} (\infty)+a_2 \Delta_{H_2} (\infty)$ and thus, $S$ is a radius of stabilization for $H$.

By triangular inequality and Equation \eqref{eq:deltalinear}, we obtain
\begin{align}
|\Delta_{H} (V)|\leq |a_1\|\Delta_{H_1} (V)|+|a_2| |\Delta_{H_2} (V)|\leq |a_1|c_1+|a_2|c_2 , \nonumber
\end{align}
for any $V\subset \mathbb{R}^d$, so $H$ has bounded add one cost, and so the result follows.
\end{proof}

Let $D$ be a finite weakly connected digraph.
Then by Lemma \ref{lem:linearsum} it suffices to show that $H_{D}$ satisfies the conditions of Proposition \ref{prop:mainclt}
to prove the CLT results in Theorem \ref{thm:mainllnclt}.

Note that $kNND(V)$ depends only on the ordering of the pairwise distances of the points in $V$,
and thus, one can easily see that $H_D$ is translation-invariant and scale-invariant.

For any point $v$ in $V$ let $h_D(V,v)$ denote the number of copies of $D$ in $kNND(V)$ containing $v$.
By Lemma \ref{lem:indegreebound} we have $d_{in}(v)\leq k \kappa'(d)$ where $\kappa'(d)$ is a constant which only depends on the dimension $d$ ($\kappa'(d)$ is defined in the Section \ref{subsec:upperbound}).
We also have $d_{out}(v)=k$, and therefore, $v$ is adjacent to at most $K=k (\kappa'(d)+1)$ arcs in $kNND(V)$.
Let $s$ be the number of vertices in $D$ i.e., $s=\text{card}(V(D))$.
Then as $D$ is weakly connected,
it is easy to verify that
\begin{align}
0\leq h_D(V,v) \leq C, \label{eq:CKs}
\end{align}
for some constant $C:=C(K,s)$ which only depends on $K$ and $s$ (i.e., $C$ is independent of $V$).
Since each copy of $D$ in $kNND(V)$ has exactly $s$ vertices we get
\begin{align}
sH_D(V)= \sum_{v\in V} h_D(V,v).  \label{eq:Hh}
\end{align}
Thus, $|H_D(V)|\leq (C/s)\cdot \text{card}(V)$ and so, $H_D$ is linearly bounded.

Now suppose that the point 0 at the origin is not in $V$.
Inserting the point 0 to $V$ may cause addition or deletion of some arcs in the $kNND$.
We definitely add arcs $(0,v)$ where $v$ is a $k$NN of 0 in $V\cup \{ 0\}$.
Also some arcs of the form $(v,0)$ are inserted whenever 0 is a $k$NN to $v$.
Notice that if $u$ is not a $k$NN of $v$, it is still not a $k$NN to $v$ after the insertion of 0.
Clearly, an arc $(v,u)$ is deleted after the addition of 0 only if
$u$ is the $k$-th NN of $v$ in $V$ and $\|v-0\| < \|v-u\|$.
And in this case, 0 becomes a $k$NN of $v$.
Let $v_1, \dots , v_p$ be points adjacent to 0 in $kNND(V\cup \{0\})$
(i.e., $(0,v_i)$ or $(v_i,0)$ is an arc in the $kNND$ after the insertion of 0)  and
recall that $p\leq K$.
Then any deleted or created copy of $D$ in the $kNND$ contains at least one of the $v_i$'s.
Since there are at most $C$ copies of $D$ containing a given vertex,
we see that $| \Delta_{H_D}(V) | \leq pC \leq  KC$ and hence, $H_D$ has a bounded add one cost.

We finally show that $H_D$ has a radius of stabilization.
But, we first construct a setting essential for the proof.
We show that there exist cones $C_1, \dots , C_m$ with 0 as their common peak such that
$x,y \in C_i\backslash \{0\}$ implies $\|x-y \| < \max \{ \|x\|, \|y\| \}$ for all $1\leq i \leq m$,
and $\cup_{i=1}^m C_i = \mathbb{R}^d$ (Lemma S in Appendix of \cite{bickel:1983}).

Note that the union of the open balls $B_{1/2}(x)$ where $x\in \partial B_1(0)$ is an open covering of $\partial B_1(0)$.
Since $\partial B_1(0)$ is compact, there exists a finite subcover, say $B_{1/2} (x_1),\dots , B_{1/2}(x_m)$.
Let $B_i'=B_{1/2}(x_i)\cap \partial B_1(0)$ and $C_i=\{ ax: x\in B_i',\ a\geq 0\}$ for all $1\leq i \leq m$.
Recall that whenever $x,y\in B_i'$, by the triangular inequality, we have $\|x-y\| <\|x-x_i\|+ \|x_i-y\|\leq 1/2+1/2=1$.
Now if $x,y\in C_i\backslash \{0\}$, then $x=ax'$ and $y=by'$ for some $a,b > 0$ and $x',y'\in B_i'$.
Assume that $a\leq b$.
Then, we have
\begin{align*}
\|x-y\| =\|a(x'-y')-(b-a)y'\| \leq a\|x'-y'\|+(b-a)\|y'\| <a+(b-a)=b=\|y\|
\end{align*}
which shows $\|x-y \| < \max \{ \|x\|, \|y\| \}$.
One can also obtain this result by a geometric argument as follows:
By construction, the angle $\widehat{x 0 y}$ is less than $60^o$ and therefore,
the edge $[xy]$ is not the largest edge of the triangle $\triangle(x0y)$, that is,
$\|x-y \| < \max \{ \|x\|, \|y\| \}$.
Moreover, it is easy to see that the union of the cones is the whole space since $x/\|x\| \in B_i'$ for some $i$
and hence $x\in C_i$, for each nonzero $x\in \mathbb{R}^d$.

Let $C_i(t)=\{ ax: x\in B_i',\ 0\leq a\leq t\}$ for all positive integers $t$ and $1\leq i \leq m$.
Given the HPP $\mathcal{P}$ of intensity 1 on $\mathbb{R}^d$,
let the random variable $T$ be the minimum $t$ such that each cone $C_i(t)$ contains at least $k+1$ points from $\mathcal{P}$, and set $S_0=T+1$.
Then $S_0$ is a.s. finite,
since $C_i=\cup_{t=1}^{\infty} C_i(t)$ contains infinitely many points from $\mathcal{P}$ almost surely for each $i$.

\begin{lemma}
\label{lem:S0}
Let $v$ be a nonzero point in $V=(\mathcal{P} \cap B_{S_0}(0) )\cup A \cup \{0\}$ for some finite $A \subset (\mathbb{R}^d \backslash B_{S_0}(0))$.
In $V$, if $0\in kNN(v)$ or $v\in kNN(0)$, then $\|v\|<S_0$.
\end{lemma}
\begin{proof}
First note that $v$ is in $C_i$ for some $i$,
and by definition of $S_0$, there exist points $u_1, \dots , u_{k+1}$ in $V$ lying in $C_i(S_0)$.
We prove the claim by contrapositive.
Suppose that $\|v\|>S_0$.
Then, $\|u_j-0\|<S_0<\|v-0\|$ for every $1\leq j \leq k+1$, and hence $v$ is not a $k$NN of 0.
Moreover, by the construction of $C_i$ and since $\|u_j\|<S_0<\|v\|$ for each $1\leq j \leq k+1$,
we have $\|v-u_j\|< \max\{\|u_j\|,\|v\|\} =\|v\|=\|v-0\|$ for every $1\leq j \leq k+1$.
Therefore, 0 is not a $k$NN to $v$.
\end{proof}

We inductively construct $S_0,S_1,S_2, \dots $ as follows:
Let the random variable $T_j$ be the minimum $t$ such that each slice $C_i(t) \backslash C_i(3S_j)$ contains at leat $k$ points from $\mathcal{P}$ and set $S_{j+1}=3T_j$.
Note that by similar arguments for $S_0$, each $S_j$ is clearly a.s. finite.
Also notice that $S_{j+1}>9S_j$  and $T_j>3S_j$ for every $j$.
\begin{lemma}
\label{lem:Sj}
Let $u$ and $v$ be points in $V=(\mathcal{P} \cap B_{S_{j+1}}(0) )\cup A $ for some finite $A \subset (\mathbb{R}^d \backslash B_{S_{j+1}}(0))$ and suppose $\| u \|<S_j$.
If $u\in kNN(v)$ or $v\in kNN(u)$, then $\|v\|<S_{j+1}$.
\end{lemma}
\begin{proof}
The proof is by contrapositive.
Assume $\|v\|>S_{j+1}$.

As $v$ is in $C_i$ for some $i$ and by construction of $S_{j+1}$ there exist points $v_1,\dots ,v_{k}$ in $C_i(T_j) \backslash C_i(3S_j)$ from $V$.
Triangular inequality implies
\begin{align}
\|u-v_s\| \leq \|u\|+\|v_s\| <S_j+T_j \label{eq:distuvs}
\end{align}
for every $1\leq s\leq k$, and
\begin{align}
\|u-v\| \geq \|v\|-\|u\| > \|v\|-S_j > 3T_j-S_j. \label{eq:distuv1}
\end{align}
Since $T_j>3S_j$, we have $3T_j-S_j> S_j+T_j $, and thus, inequalities in \eqref{eq:distuvs} and \eqref{eq:distuv1} yield
$\|u-v_s\| < \|u-v\|$ for each $1\leq s \leq k$.
Therefore, $v$ is not a $k$NN of $u$.

By \eqref{eq:distuv1} we have
\begin{align}
\|u-v\|^2 > (\|v\|-S_j)^2=\|v\|^2-2S_j\|v\|+S_j^2>\|v\|^2-2S_j\|v\| . \label{eq:distuv2}
\end{align}
Notice that $\|v_s\|>3S_j$ and $\|v\|>S_{j+1}=3T_j>3\|v_s\|$ which implies $\|v\|-\|v_s\| > 2\|v\|/3$ .
Then we get
\begin{align}
\|v_s\| (\|v\|-\|v_s\|)>3S_j \frac{2\|v\|}{3}=2S_j\|v\|. \label{eq:distvsv}
\end{align}
Inequalities in \eqref{eq:distuv2} and \eqref{eq:distvsv} yield
\begin{align}
\|v-u\|^2 >\|v\|^2-\|v_s\| (\|v\|-\|v_s\|)=\|v\|^2-\|v_s\| \|v\|+\|v_s\|^2, \label{eq:distuv3}
\end{align}
for each $1\leq s \leq k$.
Moreover, the construction of $C_i$ implies $\widehat{v0v_s}<60^o$, and hence by the cosine theorem in triangles we have
\begin{align}
\|v-v_s\|^2 < \|v\|^2+\|v_s\|^2-\|v\| \|v_s\|, \label{eq:cos}
\end{align}
for each $1\leq s \leq k$.
Then, by the inequalities in \eqref{eq:distuv3} and \eqref{eq:cos} we obtain $\|v-v_s\|< \|v-u\|$ for all $1\leq s \leq k$.
Thus, $u$ is not a $k$NN of $v$.
\end{proof}

Now let $l$ be the length of a longest path in the underlying graph of $D$.
We claim that $S_l$ is a radius of stabilization for $H_D$.
Let $V=(\mathcal{P} \cap B_{S_l}(0) )\cup A $ for some finite $A \subset (\mathbb{R}^d \backslash B_{S_l}(0))$.
Recall that, after the addition of 0,
every new or disappeared copy of $D$ contains at least one of the vertices adjacent to 0 in $kNND(V\cup \{0\})$.
Let $F$ be the set of these vertices.
Suppose a copy of $D$ contains $u\in F$ and $v$ be another vertex of this copy.
Since $D$ is weakly connected the exists a $u-v$ path in the underlying graph of $kNND(V)$,
say $u=u_0,u_1, \dots , u_s=v$.
Note that we have $s\leq l$ and $u_i\in kNN(u_{i-1})$ or $u_{i-1}\in kNN(u_i)$ for every $1\leq i \leq s$.
Lemma \ref{lem:S0} implies $\|u_1\|<S_0$.
Then, inductively Lemma \ref{lem:Sj} gives $\|u_i\| <S_i$ for each $1\leq i \leq s$ and hence,
$\|u_s\|=\|v\|<S_s\leq S_l$, that is $\|v\|<S_l$.
Therefore, any change in the $kNND$ caused by the insertion of the origin occurs in the ball $B_{S_l}(0)$ and so,
the set $A$ which is outside of this ball does not effect the add one cost.
In other words, $S_l$ is a radius of stabilization for $H_D$.

\subsection{LLN Results}
We next present c.m.c.c. results in the main theorem applying Theorem 3.2 in \cite{penrose:2002}.
Their result is on the functionals of the form $\sum_{v\in V} h(V,v)$ where
$h(V,v)$ is a functional defined for every finite point set $V\subset \mathbb{R}^d$ and $v\in V$.
Although their result is for marked point sets, the theorem is still obviously true for unmarked sets as
the unmarked point set can be viewed as a marked point set with a single mark.

The functional $h$ is said to be \emph{translation-invariant} if $h(V,v)=h(V+y,v+y)$ for every finite $V\subset \mathbb{R}^d$, $v\in V$ and $y\in \mathbb{R}^d$.
$h$ is called \emph{scale-invariant} if $h(V,v)=h(aV,av)$ for every $V\subset \mathbb{R}^d$, $v\in V$ and $a\in \mathbb{R} \backslash \{0\}$.
We say that $h$ is \emph{uniformly bounded}, if there exists a constant $c$ such that $|h(V,v)|\leq c$ for all $V\subset \mathbb{R}^d$ and $v\in V$.
The functional $h$ is called \emph{strongly stabilizing} if there exist a.s. finite random variables $S$ (a \emph{radius of stabilization} for $h$) and $h_{\infty}$ (the \emph{limit} of $h$) such that with probability 1,
\begin{align*}
h((\mathcal{P}\cap B_S(0)) \cup \{0\} \cup A, 0)=h_{\infty},
\end{align*}
for every finite set $A\subset \mathbb{R}^d \backslash B_S(0)$.

By Theorem 3.2 in \cite{penrose:2002} one can obtain the following proposition.
\begin{prop}
\label{prop:llnh}
Let $h$ be a functional defined on pairs $(V,v)$ consisting of a point set $V$ in $\mathbb{R}^d$ and an element $v$ of $V$, and $H(V)=\sum_{v\in V} h(V,v)$ for every finite $V\subset \mathbb{R}^d$.
If $H$ has a bounded add one cost, $h$ is translation invariant, scale invariant, uniformly bounded and strongly stabilizing with limit $h_{\infty}$, then as $n\rightarrow \infty$ we have
\begin{align*}
n^{-1} H(\mathcal{U}_n) \xrightarrow{c.m.c.c.} \E(h_{\infty})
\text{ and }
n^{-1} H(\mathcal{P}_n) \xrightarrow{c.m.c.c.} \E(h_{\infty}).
\end{align*}
\end{prop}

We now prove the c.m.c.c. results in the main theorem using Proposition \ref{prop:llnh}.
Since $X_n\xrightarrow{c.m.c.c.} x$ and $Y_n \xrightarrow{c.m.c.c.} y$ as $n\rightarrow \infty$ implies
$aX_n+bY_n \xrightarrow{c.m.c.c.} ax+by$ as $n\rightarrow \infty$ for any real numbers $a$ and $b$,
it suffices to prove c.m.c.c. results in the main theorem only for $H_D$.
Let $D$ be a fixed weakly connected digraph with $s$ vertices.
Let $l$ be the length of a maximal path in the underlying graph of $D$.
Recall that by \eqref{eq:Hh}, we have $H_D(V)=\sum_{v\in V} h_D(V,v)/s$ for any point set $V$.
Now set $h=h_D/s$.
We show that $h$ and $H=H_D$ satisfy the conditions of Proposition \ref{prop:llnh}.
We have already shown that $H_D$ has a bounded add one cost in the previous subsection.
Clearly, $h_D$ is both translation-invariant and scale-invariant, and so does $h$.
Recall that by \eqref{eq:CKs} we have $0\leq h_D(V,v) \leq C(K,s)$ for every point set $V$ and element $v$ of $V$, and hence, $h=h_D/s$ is uniformly bounded.
Finally, by the same arguments in the proof of strong stabilization of $H_D$, it is easy to check that
(using Lemma \ref{lem:S0} and Lemma \ref{lem:Sj}) $S_l$ is a radius of stabilization for $h$.
Also, as $h_D$ is uniformly bounded, $h_{\infty}$ (the limit of $h$) is a.s. finite and therefore, $h$ is strongly stabilizing.
So, the result follows.

\begin{remark}{\bf $k$NN graphs.} \label{rem:kNNgraphversion}
A widely studied object in statistics and probability is $k$NN graphs
(see, e.g., \cite{friedman:1983}, \cite{avram:1993}, \cite{penrose:2001} and \cite{wade2007}).
$k$NN graph of a point set is obtained by putting an edge between two points whenever one of them is a $k$NN of the other one.
In other words, $k$NN graph is the underlying graph of the $k$NN digraph.
The results we obtain for $k$NN digraphs are also valid for $k$NN graphs.
Let $G_1, \dots ,G_m$ be finite connected graphs.
For each $1\leq i \leq m$, let $H_{G_i}(V)$ denote the number of subgraphs of the $k$NN graph of $V$ which are isomorphic to $G_i$.
Then, any linear combination of $H_{G_i}$'s satisfies all the asymptotic properties of $H$ in Theorem \ref{thm:mainllnclt}.
\end{remark}

\begin{remark}{\bf Marked point processes.} \label{rem:markedppversion}
Let $(\mathcal{K}, \mathcal{F}, P)$ be a probability space.
In a marked point process with mark space $(\mathcal{K}, \mathcal{F}, P)$,
independently of other points and marks, each point is assigned a mark taking value in $\mathcal{K}$ under the distribution $P$.
In other words, marks of the points are i.i.d. with distribution $P$.
A marked graph (resp. digraph) is a graph (resp. digraph) in which each vertex has a mark.
Two marked graphs (resp. marked digraphs) are isomorphic if there exists an isomorphism between the graphs (resp. digraphs) such that the corresponding vertices under the isomorphism have the same mark.
Suppose $\mathcal{U}_{n}$ and $\mathcal{P}_{n}$ are marked binomial point processes and HPP, respectively.
Then, $D_i$'s in Theorem \ref{thm:mainllnclt} and $G_i$'s in Remark \ref{rem:kNNgraphversion} can be taken to be marked digraphs and marked graphs, respectively, and all the asymptotic results still hold.

For example, fix a positive integer $m$.
Let $\mathcal{K}=\{1,\dots ,m\}$, $\mathcal{F}=2^{\mathcal{K}}$ and $P(\text{mark is } i)=p_i>0$ for every $1\leq i \leq m$.
For every $1\leq i, j \leq m$, let $D_{ij}$ be the marked digraph with $V(D_{ij})=\{u,v\}$, $A(D_{ij})=\{(u,v)\}$ and marks of $u$ and $v$ are $i$ and $j$, respectively.
Let $N_{ij}(V)$ denote the number of marked subdigraphs of $kNND(V)$ isomorphic to $D_{ij}$ for any marked point set $V$.
In other words, $N_{ij}$ counts the arcs whose tail has mark $i$ and head has mark $j$ in the $k$NN digraph of the given marked point set.
Let $\mathcal{U}_{n}$ and $\mathcal{P}_{n}$ be as defined before with mark space $(\mathcal{K}, \mathcal{F}, P)$.
Then, $H$ in Theorem \ref{thm:mainllnclt} can be replaced with not only $N_{ij}$ but also any linear combination of $N_{ij}$'s.
When $k=1$ $N_{ij}$ is actually,
the number of $NN$ pairs whose base point is of class $i$ and $NN$ point is of class $j$,
and extensively studied for data from settings different than HPP and binomial process,
e.g., from random labelling \cite{dixon:2002} or complete spatial randomness \cite{ceyhan:cell2009}.
\end{remark}

As the proofs of Theorem \ref{thm:mainllnclt} and the statements in Remarks \ref{rem:kNNgraphversion} and \ref{rem:markedppversion} are very similar, we only gave the proof of the main theorem.

In sections 4-6, we consider special cases of $H_D(V)$ and apply our main theorem on them.

\section{Number of $j$-indegree Vertices}
\label{sec:jindegree}
In this section, we study $Q_j^{(k)}$'s.
We first identify the degenerate ones (i.e., those satisfy $Q_j^{(k)}(V)=0$ for any point set $V$).
\subsection{Upper bound on the indegree}
\label{subsec:upperbound}
Let $V$ be a finite subset of $\mathbb{R}^d$.
We show that indegrees in $kNND(V)$ are bounded above by a constant which only depends on $d$ and $k$.
\begin{definition}
The \emph{kissing number} in $\mathbb{R}^d$ is the maximum number of equal nonoverlapping spheres in $\mathbb{R}^d$ that can touch a fixed sphere of the same size and denoted as $\kappa(d)$.
\end{definition}
Currently, the exact value of $\kappa(d)$ is known only for $d=1,2,3,4,8,24$ and $\kappa(1)=2, \kappa(2)=6, \kappa(3)=12,
\kappa(4)=24, 40\leq \kappa(5)\leq 44, 72\leq \kappa(6)\leq 78, 126\leq \kappa(7)\leq 134, \kappa(8)=240, \kappa(24)=196560$ (\cite{musin:2008}, \cite{conway:1988}).
For general $d$, an asymptotic upper bound for $\kappa(d)$ is $2^{0.401d(1+o(1))}$ (\cite{Kabatiansky:1978})
and an asymptotic lower bound is $2^{0.2075d(1+o(1))}$ (\cite{wyner:1965}).
For more information about the kissing number problem see \cite{conway:1988}.

For any distinct points $x,y,z\in \mathbb{R}^d$, let $\triangle(xyz)$ denote the triangle with vertices $x,y$ and $z$.
Let $\widehat{xyz}$ denote the angle belonging to the vertex $y$ in $\triangle(xyz)$ and $[xy]$ denote the line segment (i.e., edge) with end points $x$ and $y$.

The kissing number problem can be stated in another way:
the maximum value of $m$ such that there exist $m$ points $a_1,a_2, \dots , a_m$ lying on the unit sphere in $\mathbb{R}^d$
such that $\|a_i-a_j\|\geq 1$ for all $1\leq i\neq j \leq m$ or equivalently,
$\widehat{a_i0a_j}\geq 60^o$ for all $1\leq i\neq j \leq m$ where $0$ is the  origin (also the center of the unit sphere).
To see that this is exactly the same problem, consider any arrangement of non-overlapping spheres touching a central sphere.
By a suitable translation and scaling, we may assume that the central sphere is $B_{1/2}(0)$ and the other balls are of radius $1/2$.
Clearly the centers of the touching balls are on the unit sphere and the distance between any two of them is $\geq 1$,
since they do not overlap.
Additionally, for any two points $a$ and $b$ lying on the unit sphere, we have $\|a-b\|\geq 1$ if and only if $\widehat{a0b}\geq 60^o$,
since $\triangle(a0b)$ is an isosceles triangle with two edges of length 1.

Let $\kappa':=\kappa' (d)$ be the maximum value of $m$ such that there exist $m$ points $a_1,a_2, \dots , a_m$ lying on the unit sphere in $\mathbb{R}^d$ such that $\|a_i-a_j\| > 1$ for all $1\leq i\neq j \leq m$.
Clearly $\kappa'(d)\leq \kappa(d)$ and $\kappa '(1)=2$.
Note that if in all configurations with $\kappa(d)$ points satisfying the kissing number problem there exist two points of distance 1,
then we have strict inequality $\kappa'(d)<\kappa(d)$.
For example, $\kappa(2)=6$ and it is easy to see that the only configuration with six points is the set of vertices of a regular hexagon on the unit circle.
On the other hand, vertices of a regular pentagon on the unit circle is an example with five points for $\kappa'(2)$ and hence we have $\kappa'(2)=5$.
We also have $\kappa'(3)=12$ since the vertices of a regular icosahedron on the unit ball is an example with twelve points for the kissing number problem with no pair of vertices of distance 1.
For small values of $d$, we assert that $\kappa'(d)\in \{\kappa(d)-1,\kappa(d)\}$.
For now, this assertion remains as a conjecture.
The reason why we consider $\kappa'(d)$ is explained in the following lemma.

\begin{lemma}
\label{lem:indegreebound}
Let $v$ be a vertex of $kNND(V)$ where $V$ is a random point set in $\mathbb{R}^d$ obtained by $\mathcal{P}_{n}$ or  $\mathcal{U}_{n}$.
Then, we have $d_{in} (v)\leq  \kappa'  k$ a.s..
\end{lemma}
\begin{proof}
By a convenient translation, we may suppose that $v=0$.
Let $m=d_{in} (0)$ in $kNND(V)$.
We first prove the claim for $k=1$.
If $m=0$, the claim follows trivially.
Otherwise, there exist $v_1,\dots , v_m\in V$ such that $1NN(v_i)=0$ for all $1\leq i \leq m $.
Therefore, for all $1\leq i\neq j \leq m$ we have $\| v_i -v_j \|$ is greater than both of $\| v_i \|$ and $\| v_j \|$.
In other words, $[v_i v_j]$ is the largest edge of the triangle $\triangle(v_i 0 v_j)$.
Thus, the angle $\widehat{v_i 0 v_j}$ is greater than $60^o$.

Now, for each $i$ let $a_i$ be the point on the ray $\overrightarrow{0 v_i}$ such that $\|v_i\|=1$.
Note that $\widehat{a_i 0 a_j}=\widehat{v_i 0 v_j}> 60^o$ and hence $[a_i a_j]$ is the largest edge of the isosceles triangle $\triangle(a_i 0 a_j)$.
Therefore, we obtain that $\|a_i - a_j\| >1$ for all $i\neq j$,
and hence $m\leq \kappa' $ by the definition of $\kappa' $.

For general $k$, we follow the idea for $d=2$ introduced in \cite{cuzick:1990}.
Let $v_1, \dots , v_m\in V$ such that $0\in kNN(v_i)$ for each $1\leq i \leq m$.
Let $u_1$ be the $v_i$ with largest $\|v_i\|$ and $r_1$ be $\|u_1\|$.
Delete all $v_i$'s lying in the ball $B_{r_1}(u_1)$.
Of the remaining $v_i$'s let $u_2$ be the one with largest $\|v_i\|$, and $r_2$ be $\|u_2\|$.
Delete all $v_i$'s lying in the ball $B_{r_2}(u_2)$, and continue this process until all $v_i$'s are deleted.
Let $t$ be the number of steps in this process.
By the nature of this procedure, it is clear that $r_1>r_2> \cdots >r_t$,
and for $i<j$ we have $u_j \notin B_{r_i}(u_i)$,
Thus, we obtain $\|u_i-u_j\|> r_i=\|u_i\|>\|u_j\|$ for every $i>j$
which implies that 0 is the NN of every $u_i$ for the set of points $\{u_1, \dots , u_t, 0\}$.
Therefore, we obtain $t\leq \kappa' $.
Moreover, note that we delete at most $k$ of the $v_i$'s at each step.
Because, 0 is a $k$NN to each $u_i$ and there can be at most $k-1$ points in the ball $B_{r_i}(u_i)$ other than $u_i$ itself, since 0 lies on the boundary of the ball.
Finally, we have the desired result $m\leq \kappa'  k$.
\end{proof}

\begin{lemma}
\label{lem:qjtreshold}
For every $n\geq k(\kappa' +1)+2$, we have
$P(Q^{(k)}_j(\mathcal{U}_{n})>0)$ and $P(Q^{(k)}_j(\mathcal{P}_{n})>0)$ are positive if and only if $j \leq \kappa' k$.
\end{lemma}
\begin{proof}
As an immediate consequence of Lemma \ref{lem:indegreebound}, we see that $Q^{(k)}_j$ is identically zero for every $j>\kappa' k$.

For the values of $j$ with $j\leq \kappa' k$, recall that there exist points $a_1,a_2, \dots , a_{\kappa' }$ lying on the boundary of the unit sphere
in $\mathbb{R}^d$ such that $\|a_i-a_j\| > 1$ for all $1\leq i\neq j \leq \kappa' $ by the definition of $\kappa' $.
Let $r=\min \{ \|a_i-a_j\| : 1\leq i\neq j \leq \kappa' \}$ and
$\epsilon $ be a positive number less than $\min \{(r-1)/4 , 1/4\}$.
Let $a_0=0$ and $a=(3,0,0,\dots ,0) \in \mathbb{R}^d$.

Let $V$ be a set of $n$ points such that $\text{card}(V\cap B_{\epsilon}(a_0))=1$, $\text{card}(V\cap B_{\epsilon}(a_i))=s_i\leq k$ for every $1\leq i \leq \kappa' $ and all the remaining points of $V$ are in $B_{\epsilon}(a)$.
See Figure \ref{fig:indegreej} for an example.
\begin{figure}
\centering
\scalebox{.8}{
\begin{tikzpicture}[line cap=round,line join=round,>=triangle 45,x=0.6cm,y=0.6cm]
\clip(-0.9828181818181825,-0.34101818181818455) rectangle (15.95998181818182,8.516181818181808);
\draw(4.,4.) circle (2.1cm);
\draw(7.499977145557413,4.012648342807897) circle (0.42cm);
\draw(4.955535575486824,7.3670390202638085) circle (0.42cm);
\draw(1.3363145650126254,6.27041403788739) circle (0.42cm);
\draw(0.8493974173254815,2.475630174122877) circle (0.42cm);
\draw(4.55458716405143,0.544217443549234) circle (0.42cm);
\draw(13.,4.) circle (0.42cm);
\draw(4.,4.) circle (0.42cm);
\draw (3.815181818181818,4.329581818181812) node[anchor=north west] {$v_0$};
\draw (0.0385818181818178,7.990181818181809) node[anchor=north west] {$B_{\epsilon}(a_1)$};
\draw (5.660781818181819,8.249981818181809) node[anchor=north west] {$B_{\epsilon}(a_2)$};
\draw (8.05658181818182,5.152381818181811) node[anchor=north west] {$B_{\epsilon}(a_3)$};
\draw (5.297781818181819,0.6511818181818148) node[anchor=north west] {$B_{\epsilon}(a_4)$};
\draw (3.644381818181819,5.482981818181812) node[anchor=north west] {$B_{\epsilon}(a_0)$};
\draw (-1.0031818181818177,1.7369818181818138) node[anchor=north west] {$B_{\epsilon}(a_5)$};
\draw (13.57418181818182,5.055581818181811) node[anchor=north west] {$B_{\epsilon}(a)$};
\begin{scriptsize}
\draw [fill=black] (1.5945454545454547,5.976363636363633) circle (1.5pt);
\draw [fill=black] (0.9218181818181821,6.194545454545451) circle (1.5pt);
\draw [fill=black] (1.14,6.649090909090905) circle (1.5pt);
\draw [fill=black] (4.776363636363637,7.285454545454541) circle (1.5pt);
\draw [fill=black] (5.194545454545455,7.430909090909086) circle (1.5pt);
\draw [fill=black] (5.,7.) circle (1.5pt);
\draw [fill=black] (7.321818181818182,3.794545454545452) circle (1.5pt);
\draw [fill=black] (7.630909090909092,3.703636363636361) circle (1.5pt);
\draw [fill=black] (4.430909090909091,0.30363636363636226) circle (1.5pt);
\draw [fill=black] (3.8854545454545457,3.7763636363636337) circle (1.5pt);
\draw [fill=black] (12.830909090909092,4.394545454545451) circle (1.5pt);
\draw [fill=black] (13.158181818181818,4.34) circle (1.5pt);
\draw [fill=black] (13.34,3.703636363636361) circle (1.5pt);
\draw [fill=black] (12.885454545454545,3.5945454545454516) circle (1.5pt);
\draw [fill=black] (13.,4.) circle (1.5pt);
\draw [fill=black] (12.63090909090909,4.103636363636361) circle (1.5pt);
\draw [fill=black] (12.594545454545456,3.81272727272727) circle (1.5pt);
\draw [fill=black] (13.558181818181819,3.9581818181818154) circle (1.5pt);
\draw [fill=black] (7.6854545454545455,4.230909090909088) circle (1.5pt);
\draw [fill=black] (4.830909090909092,0.7581818181818166) circle (1.5pt);
\end{scriptsize}
\end{tikzpicture}}
\caption{An illustration of $n$ points with $d_{in}(v_0)=j$ for $d=2$, $n=20$, $k=3$ and $j=11$.
Notice that $s_1=s_2=s_3=3$, $s_4=2$ and $s_5=0$.}
\label{fig:indegreej}
\end{figure}
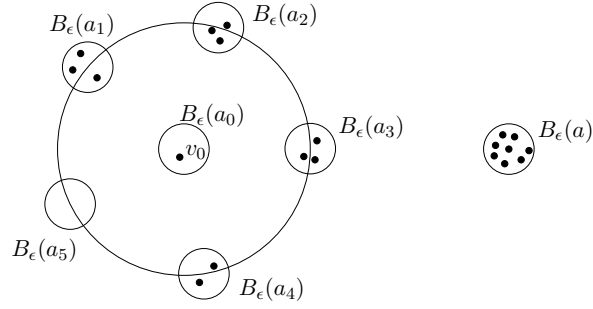
Let $v_i\in B_{\epsilon}(a_i)$ for every $i$ and $v\in B_{\epsilon}(a)$.
Then we have
\begin{align}
\|v_i-v_j\|> \|a_i-a_j\| -2\epsilon \geq r-2\epsilon >1+2\epsilon > \|v_i-v_0\| \label{eq:vivj}
\end{align}
and
\begin{align}
\|v_i-v\| > \|a_i-a\|-2\epsilon \geq 2-2\epsilon > 1+2\epsilon >\|v_i-v_0\| \label{eq:viv}
\end{align}
for all $1\leq i\neq j \leq \kappa' $.
In words, 0 is closer to $v_i$ than the points in $B_{\epsilon}(v_j)$ with $j\neq i$ and $B_{\epsilon}(a)$.
Therefore, since $s_i\leq k$ for every $i$, results in \eqref{eq:vivj} and \eqref{eq:viv} imply that
0 is a $k$NN to each $v_i$ with $i\geq 1$.
Also, for $v,w\in V\cap B_{\epsilon}(a)$ we have
\begin{align*}
\|v-w\|< 2 \epsilon < 3- 2 \epsilon < \|v-v_0\|,
\end{align*}
which implies that $w$ is closer to $v$ than $v_0$.
As $n-(s_1+\cdots +s_{\kappa' })-1 \geq n-\kappa'  k-1 \geq k+1$, $V$ contains at least $k+1$ points in $B_{\epsilon}(a)$,
and hence $v_0$ is a $k$NN to none of the points in $B_{\epsilon}(a)$.
Thus, $v_0$ is a $k$NN to exactly $j=s_1+\cdots +s_{\kappa' }$ points in $V$, that is,
$d_{in}(v_0)=j$ in $kNND(V)$.
Note that $j$ can attain any value through $0$ to $\kappa' k$ for convenient values of $s_i$'s.

Clearly, having a scaled and translated version of such a configuration described above under $\mathcal{P}_{n}$ or $\mathcal{U}_{n}$ is of positive probability, and therefore, desired result follows.
\end{proof}

\subsection{Asymptotic distribution of $Q^{(k)}_j$}
\label{subsec:cltforqjk}
We now obtain LLN and CLT results for $Q^{(k)}_j$ for every $0\leq j \leq \kappa' k$ by Theorem \ref{thm:mainllnclt}.

\begin{cor}[LLN and CLT for $Q_j^{(k)}$]
\label{thm:llnclrqj}
For every $0\leq j \leq \kappa' k$, there exist constants $q_j(d,k), \tau_j^2:=\tau_j^2 (d,k)$ and $\sigma_j^2:=\sigma_j^2(d,k)$ with $0 < \tau_j^2\leq \sigma_j^2$ such that as $n\rightarrow \infty$,
\begin{gather*}
n^{-1}Q_j^{(k)}(\mathcal{U}_n) \xrightarrow{c.m.c.c.} q_j(d,k) , \\
  n^{-1} \Var (Q_j^{(k)}(\mathcal{U}_{n})) \rightarrow \tau_j^2,\\
n^{-1/2} ( Q_j^{(k)}(\mathcal{U}_{n})-\E(Q_j^{(k)}(\mathcal{U}_{n})) ) \xrightarrow{\mathcal{L}} \mathcal{N}(0,\tau_j^2),\\
n^{-1} Q_j^{(k)}(\mathcal{P}_n) \xrightarrow{c.m.c.c.} q_j(d,k) , \\
  n^{-1} \Var (Q_j^{(k)}(\mathcal{P}_{n})) \rightarrow \sigma_j^2,\\
n^{-1/2} ( Q_j^{(k)}(\mathcal{P}_{n})-\E(Q_j^{(k)}(\mathcal{P}_{n})) ) \xrightarrow{\mathcal{L}} \mathcal{N}(0,\sigma_j^2).
\end{gather*}
\end{cor}
\begin{proof}
Let $D_i$ be the digraph with $V(D_i)=\{1,2,\dots, i+1\}$ and $A(D_i)=\{ (1,i+1),(2,i+1), \dots, (i,i+1) \}$,
for every $1\leq i \leq \kappa' k$.
Also, let $D_0$ be the digraph with only one vertex, i.e., $V(D_0)=\{1\}$ and $A(D_0)=\emptyset$.
Notice that $H_{D_0}(V)=\text{card}(V)$ for any finite point set $V$.
By principle of inclusion exclusion, it is easy to see that
\begin{align}
Q_j^{(k)} (V) = \sum_{i=j}^{\kappa' k} (-1)^{i-j} { i \choose j } H_{D_i}(V), \label{eq:qjincexc}
\end{align}
for every $0\leq j \leq \kappa' k$ and finite point set $V$.
Then the desired asymptotic results follow by Theorem \ref{thm:mainllnclt} and \eqref{eq:qjincexc}.

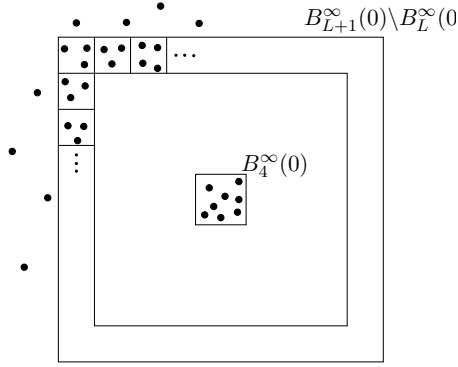
\begin{figure}
\centering
\scalebox{.8}{
\begin{tikzpicture}[line cap=round,line join=round,>=triangle 45,x=0.6cm,y=0.6cm]
\clip(-0.7851852258351781,-0.5245176631562438) rectangle (12.003799595687203,10.444402904828786);
\draw (1.,9.)-- (10.,9.);
\draw (10.,9.)-- (10.,0.);
\draw (10.,0.)-- (1.,0.);
\draw (1.,0.)-- (1.,9.);
\draw (2.,8.)-- (9.,8.);
\draw (9.,8.)-- (9.,1.);
\draw (9.,1.)-- (2.,1.);
\draw (2.,1.)-- (2.,8.);
\draw (2.,8.)-- (1.,8.);
\draw (1.,7.)-- (2.,7.);
\draw (2.,6.)-- (1.,6.);
\draw (2.,8.)-- (2.,9.);
\draw (3.,9.)-- (3.,8.);
\draw (4.,9.)-- (4.,8.);
\draw (7.58711034043659,10.05612253074082) node[anchor=north west] {$B^{\infty}_{L+1}(0)\backslash B^{\infty}_L (0)$};
\draw (5.839848657040743,6.003446126197679) node[anchor=north west] {$B^{\infty}_4(0)$};
\draw (4.8,5.2)-- (6.2,5.2);
\draw (6.2,5.2)-- (6.2,3.8);
\draw (6.2,3.8)-- (4.8,3.8);
\draw (4.8,3.8)-- (4.8,5.2);
\begin{scriptsize}
\draw [fill=black] (1.170549199482442,8.677953593664306) circle (1.5pt);
\draw [fill=black] (1.7226015492742301,8.240118971415646) circle (1.5pt);
\draw [fill=black] (2.484053066228421,8.2591552593395) circle (1.5pt);
\draw [fill=black] (2.7505610971623877,8.696989881588161) circle (1.5pt);
\draw [fill=black] (3.3406860228018855,8.278191547263356) circle (1.5pt);
\draw [fill=black] (3.3216497348780307,8.773135033283578) circle (1.5pt);
\draw [fill=black] (1.741637837198085,7.649994045776149) circle (1.5pt);
\draw [fill=black] (1.3418757907971348,7.326377151070617) circle (1.5pt);
\draw [fill=black] (1.7035652613503756,6.526853058268717) circle (1.5pt);
\draw [fill=black] (1.2657306391017158,6.545889346192572) circle (1.5pt);
\draw [fill=black] (4.25885342104692,8.5013082405611) circle (0.5pt);
\draw [fill=black] (4.5,8.5) circle (0.5pt);
\draw [fill=black] (4.7504167146820135,8.487653704626792) circle (0.5pt);
\draw [fill=black] (1.4998699622776628,5.73746407842566) circle (0.5pt);
\draw [fill=black] (1.5,5.5) circle (0.5pt);
\draw [fill=black] (1.4998699622776628,5.287108091720496) circle (0.5pt);
\draw [fill=black] (6.,4.5) circle (1.5pt);
\draw [fill=black] (1.533176510771521,6.133139800981452) circle (1.5pt);
\draw [fill=black] (1.1892738094513566,7.76667763225224) circle (1.5pt);
\draw [fill=black] (2.2661191429601217,8.658030446236445) circle (1.5pt);
\draw [fill=black] (3.748554724838356,8.137877610489694) circle (1.5pt);
\draw [fill=black] (3.748554724838356,8.71004572981112) circle (1.5pt);
\draw [fill=black] (4.900042410009558,9.382623370044977) circle (1.5pt);
\draw [fill=black] (2.887261817933763,9.40982310777573) circle (1.5pt);
\draw [fill=black] (0.419576576059164,7.4616789240189165) circle (1.5pt);
\draw [fill=black] (-0.2755672348447226,5.833573462534194) circle (1.5pt);
\draw [fill=black] (0.7056606357198758,4.54882304383711) circle (1.5pt);
\draw [fill=black] (0.05546959103643985,2.6242575515741304) circle (1.5pt);
\draw [fill=black] (1.7916129211269884,8.695966167216133) circle (1.5pt);
\draw [fill=black] (1.500194246185594,9.395370987075484) circle (1.5pt);
\draw [fill=black] (3.8315436457167507,9.861640866981716) circle (1.5pt);
\draw [fill=black] (5.0544809356466125,4.075195053015897) circle (1.5pt);
\draw [fill=black] (5.61698948166229,4.589771584065703) circle (1.5pt);
\draw [fill=black] (6.,5.) circle (1.5pt);
\draw [fill=black] (5.962116132599114,4.150137408727572) circle (1.5pt);
\draw [fill=black] (5.1793848618327365,4.824618610132642) circle (1.5pt);
\draw [fill=black] (5.5,4.) circle (1.5pt);
\draw [fill=black] (5.3042887880188605,4.308349048563329) circle (1.5pt);
\end{scriptsize}
\end{tikzpicture}}
\caption{An illustration of the event $E_1\cap E_2$ for $d=2$ and $k=2$.}
\label{fig:E1E2}
\end{figure}

To show that both $\tau_j^2$ and $\sigma_j^2$ are positive,
it suffices to prove that $\Delta_{Q_j^{(k)}} (\infty)$ is non-degenerate for every $0\leq j \leq \kappa' k$ by Proposition \ref{prop:mainclt}.
The main idea in the proof is to present two configurations with different add one costs for $Q_j^{(k)}$.
In the configurations we provide, the points near the origin are separated away from the other points, and so,
the insertion of 0 changes only the $k$NN relations between the points around the origin.

Let $L$ be a large positive integer.
Let $B^{\infty}_r(x)$ denote the corresponding $l_{\infty}$ ball, that is $B^{\infty}_r(x)=[-r,r]^d+x$.
Partition the annulus $B^{\infty}_{L+1}(0)\backslash B^{\infty}_L (0)$ into a finite collection of unit cubes.
Let $E_1$ be the event that each unit cube in the partition contains at least $k+1$ points from $\mathcal{P}$ and
$E_2$ be the event that there is no point of $\mathcal{P}$ in $B^{\infty}_{L}(0)\backslash B^{\infty}_4(0)$.
Note that whenever both $E_1$ and $E_2$ occur and $B^{\infty}_4(0)$ contains at least $k+1$ points,
any $k$NN of a point in $B^{\infty}_4(0)$ lies in $B^{\infty}_4(0)$ and
every $k$NN of a point in $\mathbb{R}^d \backslash B^{\infty}_L(0)$ lies in $\mathbb{R}^d \backslash B^{\infty}_L(0)$ for large $L$.
Therefore, in this case,
the insertion of a point at the origin can only affect the $kNND$ of the points in $B^{\infty}_4(0)$.
See Figure \ref{fig:E1E2} for an illustration.

Let $a_1, \dots , a_{\kappa' }$ and $r$ be as defined in the proof of Lemma \ref{lem:qjtreshold}.
Let $b_i=(1+r)a_i/2$ for all $1\leq i \leq \kappa' $ and $\epsilon$ be a positive real number less than $(r-1)/8$.
Now fix $j$ and let $s_1,s_2,\dots , s_{\kappa' }$ be nonnegative integers not greater than $k$ such that $s_1+\cdots + s_{\kappa' }=j$.
Let $E_3$ be the event that $B_{\epsilon}(a_i)$ contains $s_i$ points and $B_{\epsilon}(b_i)$ contains $k+1-s_i$ points
for every $1\leq i \leq \kappa' $ and there are no other points in $B^{\infty}_4(0)$.
See Figure \ref{fig:E3} for an illustration.
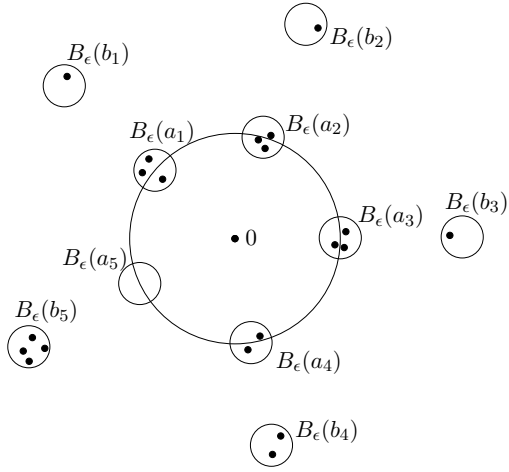
\begin{figure}
\centering
\scalebox{.8}{
\begin{tikzpicture}[line cap=round,line join=round,>=triangle 45,x=0.5cm,y=0.5cm]
\clip(-4.4590187666289,-4.000907322702685) rectangle (13.806795061843474,11.952300776119841);
\draw(3.9787900598438,4.021209940156201) circle (1.75cm);
\draw(7.478677118264148,4.0493271644460735) circle (0.35cm);
\draw(4.914126457687171,7.393915949115385) circle (0.35cm);
\draw(1.318918163236325,6.2960904850722965) circle (0.35cm);
\draw(0.8133477506193838,2.5278998628541034) circle (0.35cm);
\draw(4.512535490708562,0.5621470880196586) circle (0.35cm);
\draw (4.056090197510496,4.581514755777788) node[anchor=north west] {$0$};
\draw (0.1894648089478173,8.218362273216117) node[anchor=north west] {$B_{\epsilon}(a_1)$};
\draw (5.374513956229649,8.452018830386628) node[anchor=north west] {$B_{\epsilon}(a_2)$};
\draw (7.794877057971871,5.42292868504058) node[anchor=north west] {$B_{\epsilon}(a_3)$};
\draw (5.068695927160443,0.5763844524869007) node[anchor=north west] {$B_{\epsilon}(a_4)$};
\draw (-2.037989278259801,3.9910704980265313) node[anchor=north west] {$B_{\epsilon}(a_5)$};
\draw(6.334837395342332,11.150986524954792) circle (0.35cm);
\draw(11.530945090299781,4.077923443347411) circle (0.35cm);
\draw(5.191930275816238,-2.851364005683208) circle (0.35cm);
\draw(-2.872350129672476,0.424768061469517) circle (0.35cm);
\draw(-1.694326136572938,9.104944852729277) circle (0.35cm);
\draw (6.854432800220619,11.346482747050631) node[anchor=north west] {$B_{\epsilon}(b_2)$};
\draw (-1.8914237065546865,10.806038489299374) node[anchor=north west] {$B_{\epsilon}(b_1)$};
\draw (10.686431960101253,5.863372942791836) node[anchor=north west] {$B_{\epsilon}(b_3)$};
\draw (5.70817051340016,-1.712261435107892) node[anchor=north west] {$B_{\epsilon}(b_4)$};
\draw (-3.617604837366114,2.2601819825240185) node[anchor=north west] {$B_{\epsilon}(b_5)$};
\begin{scriptsize}
\draw [fill=black] (3.9787900598438,4.021209940156201) circle (1.5pt);
\draw [fill=black] (3.9787900598438,4.021209940156201) circle (1.5pt);
\draw [fill=black] (1.573335514389255,5.997573576519834) circle (1.5pt);
\draw [fill=black] (0.9006082416619823,6.2157553947016515) circle (1.5pt);
\draw [fill=black] (1.1187900598438,6.670300849247106) circle (1.5pt);
\draw [fill=black] (4.755153696207437,7.306664485610741) circle (1.5pt);
\draw [fill=black] (5.173335514389255,7.452119031065287) circle (1.5pt);
\draw [fill=black] (4.9787900598438,7.021209940156201) circle (1.5pt);
\draw [fill=black] (7.300608241661982,3.815755394701652) circle (1.5pt);
\draw [fill=black] (7.609699150752892,3.724846303792562) circle (1.5pt);
\draw [fill=black] (4.409699150752891,0.32484630379256296) circle (1.5pt);
\draw [fill=black] (7.664244605298346,4.2521190310652885) circle (1.5pt);
\draw [fill=black] (4.809699150752892,0.7793917583380172) circle (1.5pt);
\draw [fill=black] (-1.6013242423808693,9.41495116670284) circle (1.5pt);
\draw [fill=black] (6.737845603507963,11.026983999365367) circle (1.5pt);
\draw [fill=black] (11.120320648875746,4.129251498525416) circle (1.5pt);
\draw [fill=black] (5.499898606884265,-2.5433956746151805) circle (1.5pt);
\draw [fill=black] (-2.7539634100386428,0.7295348507078568) circle (1.5pt);
\draw [fill=black] (-2.8613994271339855,-0.058329274657992834) circle (1.5pt);
\draw [fill=black] (-2.3421253445064942,0.3714147937233797) circle (1.5pt);
\draw [fill=black] (-3.0583654584754476,0.2818847794772604) circle (1.5pt);
\draw [fill=black] (5.225850569145414,-3.1572042393111306) circle (1.5pt);
\end{scriptsize}
\end{tikzpicture}}
\caption{An illustration of the event $E_3$ for $d=2$, $k=3$, $j=11$, $s_1=s_2=s_3=3$, $s_4=2$ and $s_5=0$.}
\label{fig:E3}
\end{figure}
When $E_3$ occurs, it is easy to verify that $k$NN's of a point in $B_{\epsilon}(a_i)\cup B_{\epsilon}(b_i)$ lies in the same union.
Therefore, all the indegrees of the points in $B^{\infty}_4(0)$ are $k$.
Once the origin is inserted to the set, 0 becomes a $k$NN to the points in $B_{\epsilon}(a_i)$ for each $i$ and not a $k$NN to any point in $B_{\epsilon}(b_i)$.
Thus, the indegree of 0 is $s_1+\cdots + s_{\kappa' }=j$.
Then for $j\neq k$, we see that whenever $E_1\cap E_2\cap E_3$ occurs, $Q_j$ definitely increases.
As the event $E_1\cap E_2\cap E_3$ has a positive probability, we obtain
\begin{align}\label{eq:deltae1e2e3}
P(\Delta_{Q_j^{(k)}} (\infty)\geq 1)>0\text{ for every } j\neq k.
\end{align}

Now let $E_4$ be the event that there are $k$ points in $B_{\epsilon}(a_1)$, one point in $B_{\epsilon}(b_1)$ and no other points in $B^{\infty}_4(0)$.
After the addition of 0, the indegree of the point in $B_{\epsilon}(b_1)$ becomes 0,
the indegree of each of the points in $B_{\epsilon}(a_1)$ increases to $k+1$ and the indegree of 0 is $k$.
Since $E_1\cap E_2\cap E_4$ is an event with positive probability, we have
\begin{align}\label{eq:deltaE41}
P(\Delta_{Q_0^{(k)}} (\infty)=1)>0, P(\Delta_{Q_k^{(k)}} (\infty)=-k)>0, P(\Delta_{Q_{k+1}^{(k)}} (\infty)=k)>0,
\end{align}
and
\begin{align}\label{eq:deltaE42}
P(\Delta_{Q_j^{(k)}} (\infty)=0)>0 \text{ for each } j\notin \{0,k,k+1 \}.
\end{align}
Then by the results in \eqref{eq:deltae1e2e3} and \eqref{eq:deltaE42},
we obtain that $\Delta_{Q_j^{(k)}}$ is non-degenerate for every $0 \leq j \leq \kappa' $ with $j\notin \{0,k,k+1 \}$.

Next let $E_5$ be the event that each of $B_{\epsilon}(-0.4 a_1)$, $B_{\epsilon}(0.6 a_1)$ and $B_{\epsilon}(1.3 a_1)$  has one point, $B_{\epsilon}(0.5 a_1)$ contains $k-1$ points and there is no other point in $B^{\infty}_4(0)$, where $0<\epsilon <0.1$.
See Figure \ref{fig:E5} for an example.
One can easily see that the indegrees of the points in $B_{\epsilon}(-0.4 a_1)$, $B_{\epsilon}(0.6 a_1)$ and $B_{\epsilon}(1.3 a_1)$ are $0, k+1$ and $k$, respectively, and every point in $B_{\epsilon}(0.5 a_1)$ is of indegree $k+1$.
After the addition of the origin, the indegrees of the points in $B_{\epsilon}(-0.4 a_1)$, $B_{\epsilon}(0.6 a_1)$ and $B_{\epsilon}(1.3 a_1)$ becomes $1, k$ and $0$, respectively, and the indegree of every point in $B_{\epsilon}(0.5 a_1)$ increases to $k+2$.
Also, 0 is of indegree $k+1$.
Therefore, as $E_1\cap E_2\cap E_5$ is an event with positive probability, we have
\begin{align}\label{eq:deltaE5}
P(\Delta_{Q_0^{(k)}} (\infty)=0)>0, P(\Delta_{Q_k^{(k)}} (\infty)\geq 0)>0 \text{ and } P(\Delta_{Q_{k+1}^{(k)}} (\infty)=-(k-1))>0.
\end{align}
Then the results in \eqref{eq:deltaE41} and \eqref{eq:deltaE5} imply that $\Delta_{Q_j^{(k)}}$ is non-degenerate as well
whenever $j\in \{0,k,k+1 \}$, and we are done.

\begin{figure}
\centering
\scalebox{.8}{
\begin{tikzpicture}[line cap=round,line join=round,>=triangle 45,x=1.5cm,y=1.5cm]
\clip(0.6427619196952202,0.6119347400652031) rectangle (10.7818369328736,1.7932859320768735);
\draw(1.,1.) circle (0.3cm);
\draw(5.5,1.) circle (0.3cm);
\draw(6.,1.) circle (0.3cm);
\draw(9.5,1.) circle (0.3cm);
\draw (2.843828589352725,1.628335508574526) node[anchor=north west] {$0$};
\draw (5.920794069802475,1.6583787889793948) node[anchor=north west] {$B_{\epsilon}(0.6 a_1)$};
\draw (4.3799012106568245,1.6578335508574526) node[anchor=north west] {$B_{\epsilon}(0.5 a_1)$};
\draw (9.198391693630094,1.6570156936745394) node[anchor=north west] {$B_{\epsilon}(1.3 a_1)$};
\draw (0.6050297774389155,1.6591966461623078) node[anchor=north west] {$B_{\epsilon}(-0.4 a_1)$};
\begin{scriptsize}
\draw [fill=black] (3.,1.) circle (1.0pt);
\draw [fill=black] (1.0125762058901764,1.1563835502966686) circle (1.0pt);
\draw [fill=black] (5.5325285927174175,0.8892954547114214) circle (1.0pt);
\draw [fill=black] (5.5325285927174175,1.1050204549918135) circle (1.0pt);
\draw [fill=black] (5.357894068680911,1.002294264382103) circle (1.0pt);
\draw [fill=black] (5.943433355156257,0.8687502165894793) circle (1.0pt);
\draw [fill=black] (9.528577407435137,1.1050204549918135) circle (1.0pt);
\draw [fill=black] (5.440075021168679,1.0947478359308422) circle (1.0pt);
\draw [fill=black] (5.624982164266156,0.9612037881382187) circle (1.0pt);
\end{scriptsize}
\end{tikzpicture}}
\caption{An illustration of the event $E_5$ for $d=2$ and $k=6$.}
\label{fig:E5}
\end{figure}
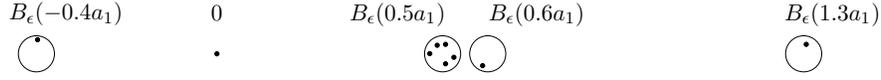
\end{proof}

Computing the exact values of the constants $q_j(d,k),\tau_j^2$ and $\sigma_j^2$ analytically is tedious, if possible at all.
For the case $k=1$, the results in \cite{newmanRT:1983} and \cite{henze:1987} imply
\begin{align}
q_j(d,1)=\frac{1}{j!}\sum _{i=0}^{\kappa' -j} \frac{1}{i!} (-1)^i b_{j+i}(d) \nonumber
\end{align}
for every $0\leq j \leq \kappa' $, where $b_0(d)=b_1(d)=1$,
\begin{align}
 b_s(d)=\underset{{\Gamma_s}}{\int \cdots \int} {\rm exp} \left[-m \left(\bigcup_{i=1}^{s} B_{\|x_i\|}(x_i) \right ) \right ]dx_1\dots dx_s, \nonumber
\end{align}
and
\begin{align}
\Gamma_s= \left \{ (x_1,\dots,x_s): x_i\in \mathbb{R}^d, \|x_i\|< \underset{1\leq l\leq s, l\neq i}{\min} \|x_i-x_l\|, 1\leq i\leq s \right \} \nonumber
\end{align}
for each $2\leq s\leq \kappa' $.
The values of $b_s(d)$ are only approximated with Monte Carlo simulations.
For the two dimensional case, \cite{cuzick:1990} provide
\begin{align}
& q_0(2,1)\approx 0.284,\ q_1(2,1)\approx 0.463 ,\ q_2(2,1) \approx 0.221 ,\nonumber
\\
& q_3(2,1) \approx 3.04\times 10^{-2} ,\ q_4(2,1)\approx 6.58 \times 10^{-4} ,\ q_5(2,1)\approx 1.90\times 10^{-7}. \nonumber
\end{align}
On the other hand, for $d=1$ and $k=1$ all the limit values are known.
In \cite{bahadir:2016} we have $q_0(1,1)=q_2(1,1)=1/4$, $q_1(1,1)=1/2$, $\tau_0^2 (1,1)=\tau_2^2 (1,1)=19/240$ and
$\tau_1^2 (1,1)=19/60$.
Using the results in \cite{bahadir:2016}, one can easily show that
$\sigma_0^2 (1,1)=\sigma_2^2 (1,1)=17/120$ and $\sigma_1^2 (1,1)=17/30$.

Moreover, limiting value of $q_j(d,k)$ as $d\rightarrow \infty$ is studied by some authors.
\cite{newmanRT:1983} focus on the case $k=1$ and show that
$$\lim_{d\rightarrow \infty} q_j(d,1) = \frac{e^{-1}}{j!}$$
for every $j\geq 0$,
whereas \cite{yao:1996} provide the answer for all $k$, that is,
$$\lim_{d\rightarrow \infty} q_j(d,k) = \frac{e^{-k}k^j}{j!}$$
for every $j\geq 0$.

\subsection{Joint Distribution of $Q_j^{(k)}$'s}
\label{subsec:asympQj}
Let $Q_{in}^{(k)}=\left ( Q_0^{(k)}, \dots , Q_{\kappa' k}^{(k)}\right)$ and
$\Sigma_{Q_{in}^{(k)}} (\mathcal{U}_n)$ (resp. $\Sigma_{Q_{in}^{(k)}} (\mathcal{P}_n)$) be the covariance matrix of $Q_{in}^{(k)}(\mathcal{U}_n)$ (resp. $Q_{in}^{(k)}(\mathcal{P}_n)$).
Note that for any two random variables $X$ and $Y$, we have $\Cov(X,Y)=( \Var(X+Y)-\Var(X-Y) )/4$.
Since any linear combination of $Q_j^{(k)}$'s is a linear combination of $H_{D_i}$'s defined in the proof of Theorem \ref{thm:llnclrqj},
Theorem \ref{thm:mainllnclt} implies that $\Cov(Q_{j_1}^{(k)}, Q_{j_2}^{(k)})/n$ converges to a constant as $n\rightarrow \infty$ for any $0\leq j_1,j_2 \leq \kappa' k$.
Therefore, there exist constant $(1+\kappa' k) \times (1+\kappa' k)$ matrices $\Sigma_{\mathcal{U}}:=\Sigma_{\mathcal{U}}(k,d)$ and $\Sigma_{\mathcal{P}}:=\Sigma_{\mathcal{P}}(k,d)$ such that
\begin{align}
\Sigma_{Q_{in}^{(k)}} (\mathcal{U}_n)/n \rightarrow \Sigma_{\mathcal{U}} \text{ and }
\Sigma_{Q_{in}^{(k)}} (\mathcal{P}_n)/n \rightarrow \Sigma_{\mathcal{P}}, \label{eq:limmat}
\end{align}
as $n\rightarrow \infty$.
Similarly, we see that any linear combination of $Q_j^{(k)}/\sqrt{n}$'s converges in law to a normal variable, and therefore,
Cramer-Wold device yields the following corollary.

\begin{cor}
\label{cor:qjjoint}
Let $ \Sigma_{\mathcal{U}}$ and $\Sigma_{\mathcal{P}}$ be the limiting matrices in \eqref{eq:limmat}.
Then, as $n\rightarrow \infty$ we have
\begin{align*}
n^{-1/2} (Q_{in}^{(k)}(\mathcal{U}_n)-\E(Q_{in}^{(k)}(\mathcal{U}_n))) \xrightarrow{\mathcal{L}} \mathcal{N}(\bf{0}, \Sigma_{\mathcal{U}})
\end{align*}
and
\begin{align*}
n^{-1/2} (Q_{in}^{(k)}(\mathcal{P}_n)-\E(Q_{in}^{(k)}(\mathcal{P}_n))) \xrightarrow{\mathcal{L}} \mathcal{N}(\bf{0}, \Sigma_{\mathcal{P}}),
\end{align*}
where $\bf{0}$ is the zero vector in $\mathbb{R}^{1+\kappa' k}$ and $\mathcal{N}(\mathbf{\mu}, \Sigma)$ stands for the multivariate normal variable with mean vector $\mathbf{\mu}$ and covariance matrix $\Sigma$.
\end{cor}

Next, we study the ranks of the covariance matrices $\Sigma_{Q_{in}^{(k)}} (\mathcal{U}_n)$ and $\Sigma_{Q_{in}^{(k)}} (\mathcal{P}_n)$.
Since $Q_j^{(k)}(V)$ counts the number of vertices in $kNND(V)$ with indegree $j$, we have
\begin{align}
\label{eq:qjsum}
\text{card}(V)=\sum_{j=0}^{\kappa' k} Q_j^{(k)}(V).
\end{align}
As the number of arcs is equal to the sum of all the indegrees in a digraph, considering the $kNND(V)$ gives
\begin{align}
\label{eq:sumindegree}
\text{card}(V)k=\sum_{j=0}^{\kappa' k} j Q_j^{(k)}(V).
\end{align}
The equations in \eqref{eq:qjsum} and \eqref{eq:sumindegree} appears to be the only linear relations
between $Q_j^{(k)}(V)$'s and $\text{card}(V)$.
Combining the results in these equations yields
\begin{align}
\label{eq:qjdependence}
 \sum_{j=0}^{\kappa' k} (j-k) Q_j^{(k)}=0.
\end{align}
Note that equation in \eqref{eq:qjdependence} provides a non-trivial linear dependence relation for $Q_j^{(k)}$'s .

Now, suppose that $\Sigma_{Q_{in}^{(k)}}(\mathcal{X}_n)\mathbf{a}=0$ for some $(1+\kappa' k)\times 1$ real vector $\mathbf{a}$,
where $\mathcal{X}_n$ is $\mathcal{U}_n$ or $\mathcal{P}_n$.
Then, note that
\begin{align*}
0=\mathbf{a}^t \Sigma_{Q_{in}^{(k)}}(\mathcal{X}_n)\mathbf{a}=\Var \left(\sum_{j=0}^{\kappa' k} a_j Q_j^{(k)}(\mathcal{X}_n) \right),
\end{align*}
where $\mathbf{a}^t=(a_0, \dots, a_{\kappa' k})$ is the transpose of the vector $\mathbf{a}$.
So, we obtain that $\sum_{j=0}^{\kappa' k} a_j Q_j^{(k)}(\mathcal{X}_n)$ is non-random as its variance is 0.
Notice that letting $\mathbf{a}^t$ to be any scalar multiple of $(k, k-1, \cdots, 1,0,-1, \cdots ,(k-\kappa' k))$
satisfies the assumption for $\mathbf{a}$ because of the equation \eqref{eq:qjdependence}.

Recall that $\text{card}(\mathcal{U}_{n})=n$ which is non-random,
and therefore we can also take $\mathbf{a}^t=(1,1,\dots ,1)$ by \eqref{eq:qjsum}.
Thus, the rank of $\Sigma_{Q_{in}^{(k)}}(\mathcal{U}_n)$ is at most $(1+\kappa' k)-2=\kappa' k-1$.
But, on the other hand, $\text{card}(\mathcal{P}_{n})$ has a Poisson distribution with mean and variance equal to $n$.
Hence, the equation \eqref{eq:qjsum} or \eqref{eq:sumindegree} does not provide another example for $\mathbf{a}$,
and so we can only state that the rank of $\Sigma_{Q_{in}^{(k)}}(\mathcal{P}_n)$ is at most $(1+\kappa' k)-1=\kappa' k$.
We strongly believe that the upper bounds we provided for the rank of the covariance matrices are actually
equalities,
for every $n\geq k(\kappa' +1)+2$.
Furthermore, the limiting matrices $\Sigma_{\mathcal{U}}$ and $\Sigma_{\mathcal{P}}$ seem to have the same corresponding ranks.
Yet, these assertions currently remain as conjectures.

\section{Number of Shared $k$NN's}
\label{sec:shared}
In this section we study the asymptotic distribution of $Q^{(k)}$.
Recall that $Q^{(k)}=H_D$ where $D$ is the digraph with vertex set $\{1,2,3\}$ and arc set $\{(1,2),(3,2)\}$.
Thus, we can apply Theorem \ref{thm:mainllnclt} for $Q^{(k)}$.
Moreover,
the events $E_1\cap E_2\cap E_4$ and $E_1\cap E_2\cap E_5$ described in the proof of Theorem \ref{thm:llnclrqj} give
\begin{align}
P(\Delta_{Q^{(k)}} (\infty)=k^2)>0 \text{ and } P(\Delta_{Q^{(k)}} (\infty)=k^2-1)>0, \nonumber
\end{align}
respectively, and hence, we get $\Delta_{Q^{(k)}} (\infty)$ is non-degenerate.
So, by Proposition \ref{prop:mainclt}, the limiting variance values are positive, and we obtain the following result.

\begin{cor}[LLN and CLT for $Q^{(k)}$]
\label{thm:llnclrq}
There exist constants $q(d,k), \tau^2_Q:=\tau^2_Q (d,k)$ and $\sigma^2_Q:=\sigma^2_Q(d,k)$ with $0 < \tau^2 _Q\leq \sigma^2_Q$ such that as $n\rightarrow \infty$,
\begin{gather*}
n^{-1}Q^{(k)}(\mathcal{U}_n) \xrightarrow{c.m.c.c.} q(d,k) , \\
  n^{-1} \Var (Q^{(k)}(\mathcal{U}_{n})) \rightarrow \tau^2_Q,\\
n^{-1/2} ( Q^{(k)}(\mathcal{U}_{n})-\E(Q^{(k)}(\mathcal{U}_{n})) ) \xrightarrow{\mathcal{L}} \mathcal{N}(0,\tau^2_Q),\\
n^{-1} Q^{(k)}(\mathcal{P}_n) \xrightarrow{c.m.c.c.} q(d,k) , \\
  n^{-1} \Var (Q^{(k)}(\mathcal{P}_{n})) \rightarrow \sigma^2_Q,\\
n^{-1/2} ( Q^{(k)}(\mathcal{P}_{n})-\E(Q^{(k)}(\mathcal{P}_{n})) ) \xrightarrow{\mathcal{L}} \mathcal{N}(0,\sigma^2_Q).
\end{gather*}
\end{cor}

For $d=1$, Theorem 3.2 in \cite{schiling:1986} implies $q(1,k)=k^2/2-k/4$ for each $k$.
By \cite{bahadir:2016}, we also have $\tau^2_Q(1,1)=19/240$ and $\sigma^2_Q(1,1)=17/120$.

For $d\geq 2$, we only have numerical approximations for $q(d,k)$.
For example, \cite{cuzick:1990} provides
\begin{align}
q(2,1)\approx 0.3166, q(2,2)\approx 1.58685, q(2,3)\approx 3.84845 . \label{eq:CEqapprox}
\end{align}
On the other hand, the results of \cite{schiling:1986} give
\begin{align}
q(2,1)\approx 0.315, q(2,2)\approx 1.575, q(2,3)\approx 3.82 ,\label{eq:Sqapprox}\\
q(3,1)\approx 0.355, q(3,2)\approx 1.645, q(3,3) \approx 3.93. \nonumber
\end{align}
Notice that the results in \eqref{eq:CEqapprox} and \eqref{eq:Sqapprox} slightly differ.
Monte Carlo estimations we derived for $q(2,k)$ for $k=1,2,3$ are closer to the ones in \eqref{eq:CEqapprox}.
For $k=1,\dots , 5$, exact value of $q(1,k)$ and the value of $q(2,k)$ obtained in \cite{cuzick:1990} are presented in Table \ref{tab:roundvalqdk}.

\begin{table}[t]
\center
\caption{Values of $q(d,k)$ for $d=1,2$ and $k=1,\dots ,5$.}
\label{tab:roundvalqdk}
\begin{tabular}{cccccc}
\hline \hline \\
$d$ & $k=1$ & $k=2$ & $k=3$ & $k=4$ & $k=5$  \\
\cline{1-6}\\
 1 & 0.25 & 1.5 & 3.75 & 7 & 11.25 \\
2 & 0.3166 & 1.5868 & 3.8484 & 7.1079 & 11.3667 \\
\hline
\end{tabular}
\end{table}

Furthermore, results in \cite{schiling:1986} implies
$$\lim_{d \rightarrow \infty} q(d,k)= \frac{k^2 }{2},$$
for any $k$.
Let $V=\{v_1,\dots,v_s\}$ and $d_i$ denote the indegree of $v_i$ in $kNND(V)$.
Then, it is easy to see that the number of shared $k$NN's is $\sum_{i=1}^s d_i(d_i-1)/2$.
Moreover, a double counting argument for the number of arcs in the $k$NND gives $\sum_{i=1}^s d_i=sk$.
Thus, we get $\sum_{i=1}^s d_i(d_i-1)/2=s(k^2-k)/2+\sum_{i=1}^s (d_i-k)^2/2 \geq s(k^2-k)/2$
which yields
\begin{align}
Q^{(k)}(V) \geq \frac{\text{card}(V) \cdot (k^2-k)}{2}, \label{eq:lowboundQQ}
\end{align}
for any finite point set $V$.
By \eqref{eq:lowboundQQ}, one can easily obtain
\begin{align}
q(d,k)\geq \frac{k(k-1)}{2}, \label{eq:lowerboundqdk}
\end{align}
for every $d$ and $k$.

Recall that
\begin{align*}
Q^{(k)}=\sum_{j=0}^{\kappa' k} { j \choose 2} Q_j^{(k)},
\end{align*}
and so, we have
\begin{align*}
q(d,k)=\sum_{j=0}^{\kappa' k} { j \choose 2} q_j(d,k).
\end{align*}
Then, it is easy to verify that $q(d,1)=b_2(d)/2$ where $b_2(d)$ is as defined in Section \ref{subsec:cltforqjk}.
Yet, for $d\geq 2$, we only have approximation of $b_2(d)$ based on Monte Carlo simulations.

\section{Number of Reflexive $k$NN's}
\label{sec:reflexive}
In this section, we study the asymptotic behavior of $R^{(k)}$.

\begin{cor}[LLN and CLT for $R^{(k)}$]
\label{thm:llnclrq}
There exist constants $r(d,k), \tau^2_R:=\tau^2_R (d,k)$ and $\sigma^2_R:=\sigma^2_R(d,k)$ with $0 < \tau^2 _R\leq \sigma^2_R$ such that as $n\rightarrow \infty$,
\begin{gather*}
n^{-1}R^{(k)}(\mathcal{U}_n) \xrightarrow{c.m.c.c.} r(d,k) , \\
  n^{-1} \Var (R^{(k)}(\mathcal{U}_{n})) \rightarrow \tau^2_R,\\
n^{-1/2} ( R^{(k)}(\mathcal{U}_{n})-\E(R^{(k)}(\mathcal{U}_{n})) ) \xrightarrow{\mathcal{L}} \mathcal{N}(0,\tau^2_R),\\
n^{-1} R^{(k)}(\mathcal{P}_n) \xrightarrow{c.m.c.c.} r(d,k) , \\
  n^{-1} \Var (R^{(k)}(\mathcal{P}_{n})) \rightarrow \sigma^2_R,\\
n^{-1/2} ( R^{(k)}(\mathcal{P}_{n})-\E(R^{(k)}(\mathcal{P}_{n})) ) \xrightarrow{\mathcal{L}} \mathcal{N}(0,\sigma^2_R).
\end{gather*}
\end{cor}
\begin{proof}
Recall that $R^{(k)}=H_D$ where $D$ is the digraph with vertex set $\{1,2\}$ and arc set $\{(1,2),(2,1)\}$, and therefore,
Theorem \ref{thm:mainllnclt} yields the asymptotic results.

Furthermore, we also show that $\Delta_R(\infty)$ is non-degenerate which implies $0 < \tau_R^2 $ and $0 < \sigma_R^2$.
Both of the events $E_1\cap E_2\cap E_4$ and $E_1\cap E_2\cap E_5$ defined in the proof of Theorem \ref{thm:llnclrqj} yield
\begin{align}\label{eq:deltaR1}
 P(\Delta_{R^{(k)}} (\infty)=0)>0.
\end{align}
Letting $s_1=s_2=k$ and $s_i=0$ for all $i\geq 3$ for the event $E_3$ in the proof of Theorem \ref{thm:llnclrqj} gives
\begin{align}\label{eq:deltaR2}
 P(\Delta_{R^{(k)}} (\infty)=k)>0.
\end{align}
Then by \eqref{eq:deltaR1} and \eqref{eq:deltaR2} we see that $\Delta_{R^{(k)}} (\infty)$ is not degenerate.
\end{proof}

By the results in \cite{henze:1987}, \cite{pickard:1982} and \cite{schiling:1986}, we obtain
\begin{align*}
r(d,k)=\sum_{s=1}^k \sum_{t=1}^k r_d(s,t),
\end{align*}
where
\begin{align*}
r_d(s,t)=\frac{\omega(d)}{2} \sum_{i=0}^{\min \{s-1,t-1\}} \frac{(s+t-i-2)!}{i! (s-i-1)! (t-i-1)!} (2\omega(d)-1)^i
(1-\omega(d))^{s+t-2i-2}
\end{align*}
and
\[
 \omega(d) =
  \begin{cases}
   \displaystyle \left[ \frac{3}{2}+\frac{1}{2}\sum_{i=1}^m \frac{1\cdot 3  \cdots   (2i-1)}{2\cdot 4 \cdots (2i)} \left(\frac{3}{4} \right)^i \right ]^{-1} & \text{if } d=2m+1, \\
   \displaystyle \left[ \frac{4}{3}+\frac{\sqrt{3}}{2\pi} \left (1+\sum_{i=1}^{m-1} \frac{2\cdot 4  \cdots   (2i)}{3\cdot 5 \cdots  (2i+1)} \left(\frac{3}{4} \right)^i \right) \right ]^{-1} & \text{if } d=2m.
  \end{cases}
\]
From a geometric point of view, $\omega(d)$ is the volume of a unit sphere in $\mathbb{R}^d$ divided by the volume of the
union of two such spheres whose centres are separated by a distance 1.
The rounded values of $r(d,k)$ for $d=1,2$ and $k=1,\dots ,5$ are presented in Table \ref{tab:roundvalrdk}.
\begin{table}[t]
\center
\caption{Values of $r(d,k)$ for $d=1,2$ and $k=1,\dots ,5$.}
\label{tab:roundvalrdk}
\begin{tabular}{cccccc}
\hline \hline \\
$d$ & $k=1$ & $k=2$ & $k=3$ & $k=4$ & $k=5$  \\
\cline{1-6}\\
 1 & 0.3333 & 0.7407 & 1.1728 & 1.6168 & 2.0680 \\
 2 & 0.3107 & 0.7105 & 1.1365 & 1.5751 & 2.0215 \\
 \hline
\end{tabular}
\end{table}

However, we have almost nothing about the exact values of limiting variances $\tau^2_R$ and $\sigma^2_R$.
We only have $\tau^2_R (1,1)=2/45$ and $\sigma^2_R(1,1)=7/45$ by the results in \cite{bahadir:2016}.

Finally, we present an interesting connection between reflexive pairs and components of a $1$NN digraph which is also stated in \cite{eppstein:1997} and \cite{enns:1999}.
\begin{prop}
Let $V$ be a finite point set including at least two points such that pairwise distances between the points of $V$ are all distinct.
Then, each weakly connected component of 1NN digraph of $V$ contains exactly one reflexive pair and hence,
number of weakly connected components of $1NND(V)$ is $R{[1]}(V)$.
\end{prop}
\begin{proof}
Construct the $1$NN digraph of $V$.
We first show that any weakly connected component of $1NND(V)$ contains at least one reflexive pair.
Note that any weakly connected component contains at least two points and therefore,
it also contains at least one arc.
Pick any component and the shortest arc in it
(Here by \emph{length} of an arc we refer to the distance between the endpoints of the arc, and the shortest arc is the one with the minimum length.).
As there are finitely many arcs in the component, the shortest arc exists, say $(u,v)$.
Then, by definition we see that $NN$ of $u$ is $v$.
Let $NN$ of $v$ be $w$.
Clearly the arc $(v,w)$ belongs to this component and hence,
we get $\|v-w\| \geq \|v-u\|$ by the choice of $(u,v)$.
On the other hand, considering the definition of NN for $v$ yields $\|v-w\| \leq \|v-u\|$.
Thus, by the uniqueness of the NN of any point we obtain that $w=u$ which implies that $\{u,v\}$ is a reflexive pair in this component.

Next suppose that a component contains two reflexive pairs $\{u,v\}$ and $\{w,z\}$.
As each point has a unique NN,
we have $\{u,v\}\cap \{w,z\}=\emptyset$.
In the underlying graph of $1NND(V)$,
consider a shortest path with one end point from $\{u,v\}$ and the other one from $\{w,z\}$,
say $v_1, \dots, v_s$.
Since the edges corresponding to the reflexive pairs appear in the same component,
such paths exist and as the graph is finite, a shortest one is well defined.
Without loss of generality,
we may assume that $v_1=v$ and $v_s=z$.
Note that by the choice of the path we have $v_2\neq u$.
Moreover, since we have the edge $v_1v_2$ in the underlying graph,
we see that 1NN of $v_1$ is $v_2$ or 1NN of $v_2$ is $v_1$.
On the other hand, we have $v_1=v$ and 1NN of $v$ is $u \neq v_2$, and hence,
we obtain that 1NN of $v_2$ is $v_1$.
In the same manner, we see that one of $v_2$ and $v_3$ is the 1NN of the other one, and therefore,
we get 1NN of $v_3$ is $v_2$ since 1NN of $v_2$ is $v_1$ and $v_1\neq v_3$.
By induction on the indices of $v_i$'s we obtain that the 1NN of $v_s$ is $v_{s-1}$, that is, $v_{s-1}$ is the 1NN of $z$. Then, the uniqueness of 1NN implies $w=v_{s-1}$ which contradicts the choice of the path.
Thus, each component has exactly one reflexive pair and therefore, the result follows.
\end{proof}

\section{The Case of $k=1$}
\label{sec:depRQQj}
In this section, we study the case $k=1$ where
we can show that $R^{(1)}(\mathcal{U}_n)$, $Q^{(1)}(\mathcal{U}_n)$ and $Q_j^{(1)}(\mathcal{U}_n)$ are pairwise dependent.
For simplicity in the notation, let $R$, $Q$ and $Q_j$ denote
$R^{(1)}(\mathcal{U}_n)$, $Q^{(1)}(\mathcal{U}_n)$ and $Q_j^{(1)}(\mathcal{U}_n)$, respectively.

We first show that $R$ and $Q_j$'s are pairwise stochastically dependent for large $n$.
We use the simple fact that if the random variables $X$ and $Y$ satisfy $P(X\in A)>0$, $ P(Y\in B)>0$ and $P(X\in A,\ Y\in B)=0$ for some Borel sets $A$ and $B$, then $X$ and $Y$ are dependent.
Also recall that for $n\geq 4$,
if the sample points $x$ and $y$ are sufficiently close to each other and far from the other $n-2$ sample points,
then $\{x,y\}$ forms a reflexive pair with each of indegree 1.
In addition, for $n\geq 5$,
if the sample points $x,y,z$ are sufficiently close to each other and far from the remaining $n-3$ sample points,
then the indegrees of $x,y,z$ in the $NND$ are $0,1,2$ in some order.

\begin{prop}
\label{thm:RdepQj}
For every $n\geq \max \{9, \kappa' +3\}$, $R$ and $Q_j$ are dependent for all
$0\leq j \leq \kappa' $.
Moreover, $R$ and $Q$ are dependent for every $n\geq 6$.
\end{prop}
\begin{proof}
We first prove the dependence of $R$ and $Q_j$'s.
The proof is based on the parity of $n$.
First assume that $n$ is an even positive integer.
It is easy to see that $R=n/2$ if and only if sample points consist of $n/2$ pairs which are pairwise far enough
(i.e., members of each pair is NN to each other and each pair is sufficiently far from other such pairs).
In this case, each point is of indegree 1.
Therefore, we have
\begin{align}\label{eq:RdepQj}
\left \{R=n/2  \right \} \subseteq \{Q_1=n\} \cap \left( \bigcap_{\substack{0\leq j \leq \kappa' \\ j\neq 1}} \{Q_j=0\} \right).
\end{align}
The event that each of $n/2$ sufficiently small balls in the region contains exactly 2 sample points has a positive probability and therefore,
$P(R=n/2)>0$ (See Figure \ref{fig:Rcases} (a)).
Moreover, by Lemma \ref{lem:qjtreshold} we have $P(Q_j\geq 1)>0$ for every $0\leq j\leq \kappa' $.
By \eqref{eq:RdepQj} we have $\{R=n/2\}\cap \{Q_j\geq 1\}=\emptyset$
for each $0\leq j\neq 1 \leq \kappa' $,
and hence, we obtain that $R$ and $Q_j$ are dependent for every $0\leq j\neq 1 \leq \kappa' $.

Next, note that as $\sum_{j=0}^{\kappa' } Q_j=n$,
we have $\{Q_0>0\}\subseteq \{Q_1<n\}$.
We know that $P(Q_0>0)>0$ by Lemma \ref{lem:qjtreshold}, and therefore $P(Q_1<n)>0$.
Moreover, the result in \eqref{eq:RdepQj} gives $\{R=n/2\}\cap \{Q_1<n\}=\emptyset$,
and thus, $R$ and $Q_1$ are dependent as well.

\begin{figure}
\centering
\scalebox{.8}{
\begin{tikzpicture}[line cap=round,line join=round,>=triangle 45,x=1.0cm,y=1.0cm]
\clip(0.6049778637383369,0.4011697885486583) rectangle (14.177954283749214,5.154109374482612);
\draw (1.,5.)-- (5.,5.);
\draw (5.,5.)-- (5.,1.);
\draw (5.,1.)-- (1.,1.);
\draw (1.,1.)-- (1.,5.);
\draw (5.5,5.)-- (9.5,5.);
\draw (9.5,5.)-- (9.5,1.);
\draw (9.5,1.)-- (5.5,1.);
\draw (5.5,1.)-- (5.5,5.);
\draw (10.,5.)-- (14.,5.);
\draw (14.,1.)-- (14.,5.);
\draw (14.,1.)-- (10.,1.);
\draw (10.,1.)-- (10.,5.);
\draw(1.4229419717426361,4.38008626048973) circle (0.3cm);
\draw(3.190801410348735,4.232764640605889) circle (0.3cm);
\draw(4.418481576047415,4.318702252204796) circle (0.3cm);
\draw(3.804641493198075,3.3488349213028408) circle (0.3cm);
\draw(3.4854446501164182,2.452628400342806) circle (0.3cm);
\draw(4.553526394274269,2.403521193714859) circle (0.3cm);
\draw(3.595935865029299,1.4827610694408504) circle (0.3cm);
\draw(2.331425294359659,2.2561995738310174) circle (0.3cm);
\draw(1.9017372363651213,1.4827610694408504) circle (0.3cm);
\draw(1.3247275584867428,2.5841655609533793) circle (0.3cm);
\draw(2.073027849960224,3.519540925295229) circle (0.3cm);
\draw(6.165295068955826,4.548453826071264) circle (0.3cm);
\draw(7.100670433297678,3.706615998163599) circle (0.3cm);
\draw(8.012661413530983,4.396455329365713) circle (0.3cm);
\draw(8.936344585818562,3.6715394220007798) circle (0.3cm);
\draw(8.,3.) circle (0.3cm);
\draw(7.7671253803912474,2.280168567542278) circle (0.3cm);
\draw(8.737577320895918,1.80078869331708) circle (0.3cm);
\draw(7.276053314111775,1.3564853952547014) circle (0.3cm);
\draw(6.773289055778029,2.3386295278136435) circle (0.3cm);
\draw(5.896374651707544,2.853085978201661) circle (0.3cm);
\draw(5.861298075544724,1.5552526601773444) circle (0.3cm);
\draw(10.503098321091164,4.291225600877255) circle (0.3cm);
\draw(10.81878750655654,3.215543931884128) circle (0.3cm);
\draw(10.491406129036891,1.508483891960252) circle (0.3cm);
\draw(12.408925625937687,1.812480885371353) circle (0.3cm);
\draw(11.648933142409932,2.4321670642478286) circle (0.3cm);
\draw(12.572616314697513,3.624770653783687) circle (0.3cm);
\draw(11.894469175549668,4.349686561148621) circle (0.3cm);
\draw(13.554760447256456,4.396455329365713) circle (0.3cm);
\draw(13.484607294930818,2.4906280245191943) circle (0.3cm);
\draw(13.589837023419276,1.4500229316888862) circle (0.3cm);
\draw (2.6,1.01) node[anchor=north west] {$(a)$};
\draw (7.1,1.01) node[anchor=north west] {$(b)$};
\draw (11.6,1.01) node[anchor=north west] {$(c)$};
\begin{scriptsize}
\draw [fill=black] (1.3481119425952892,4.408147521419987) circle (1.5pt);
\draw [fill=black] (1.5351870154636595,4.291225600877255) circle (1.5pt);
\draw [fill=black] (3.090248558681988,4.23276464060589) circle (1.5pt);
\draw [fill=black] (3.2656314394960853,4.244456832660163) circle (1.5pt);
\draw [fill=black] (4.3179287243806685,4.337994369094347) circle (1.5pt);
\draw [fill=black] (4.505003797249039,4.291225600877255) circle (1.5pt);
\draw [fill=black] (3.8268566581011965,3.238928315992674) circle (1.5pt);
\draw [fill=black] (3.6982425455041916,3.3792346206439516) circle (1.5pt);
\draw [fill=black] (2.213334154611502,3.5663096935123213) circle (1.5pt);
\draw [fill=black] (1.9210293532546734,3.4376955809153174) circle (1.5pt);
\draw [fill=black] (1.2896509823239235,2.666010905333291) circle (1.5pt);
\draw [fill=black] (1.4533416710837475,2.5257046006820136) circle (1.5pt);
\draw [fill=black] (2.295179498991414,2.1398622628910005) circle (1.5pt);
\draw [fill=black] (2.4471779956969653,2.3853982960307363) circle (1.5pt);
\draw [fill=black] (2.061335657905951,1.5201760840145249) circle (1.5pt);
\draw [fill=black] (1.8625683929833077,1.438330739634613) circle (1.5pt);
\draw [fill=black] (3.499475280581548,1.4617151237431594) circle (1.5pt);
\draw [fill=black] (3.6631659693413723,1.3915619714175207) circle (1.5pt);
\draw [fill=black] (3.534551856744368,2.4321670642478286) circle (1.5pt);
\draw [fill=black] (3.3474767838759973,2.3970904880850092) circle (1.5pt);
\draw [fill=black] (4.341313108489215,2.3853982960307363) circle (1.5pt);
\draw [fill=black] (4.645310101900317,2.280168567542278) circle (1.5pt);
\draw [fill=black] (6.270524797444284,4.618606978396903) circle (1.5pt);
\draw [fill=black] (6.1536028769015525,4.431531905528533) circle (1.5pt);
\draw [fill=black] (6.901903168375034,3.7533847663806914) circle (1.5pt);
\draw [fill=black] (7.1708235856233165,3.6715394220007798) circle (1.5pt);
\draw [fill=black] (7.930816069151071,4.326302177040074) circle (1.5pt);
\draw [fill=black] (8.211428678453627,4.361378753202894) circle (1.5pt);
\draw [fill=black] (8.85449924143865,3.613078461729414) circle (1.5pt);
\draw [fill=black] (9.123419658686933,3.624770653783687) circle (1.5pt);
\draw [fill=black] (8.176352102290807,2.9349313225815727) circle (1.5pt);
\draw [fill=black] (7.837278532716886,3.0986220113413965) circle (1.5pt);
\draw [fill=black] (6.7615968637237565,2.268476375488005) circle (1.5pt);
\draw [fill=black] (6.574521790855386,2.3970904880850092) circle (1.5pt);
\draw [fill=black] (6.024988764304548,2.8180094020388418) circle (1.5pt);
\draw [fill=black] (5.779452731164812,2.771240633821749) circle (1.5pt);
\draw [fill=black] (5.6976073867849,1.508483891960252) circle (1.5pt);
\draw [fill=black] (5.9431434199246365,1.4850995078517055) circle (1.5pt);
\draw [fill=black] (7.276053314111775,1.2863322429290625) circle (1.5pt);
\draw [fill=black] (7.252668930003229,1.5435604681230712) circle (1.5pt);
\draw [fill=black] (7.77881757244552,2.1515544549452734) circle (1.5pt);
\draw [fill=black] (7.7671253803912474,2.4321670642478286) circle (1.5pt);
\draw [fill=black] (8.796038281167284,1.8943262297512649) circle (1.5pt);
\draw [fill=black] (8.74926951295019,1.707251156882895) circle (1.5pt);
\draw [fill=black] (8.562194440081822,1.8358652694798994) circle (1.5pt);
\draw [fill=black] (10.760326546285173,3.157082971612762) circle (1.5pt);
\draw [fill=black] (10.90063285093645,3.4026190047524976) circle (1.5pt);
\draw [fill=black] (10.643404625742441,4.314609984985801) circle (1.5pt);
\draw [fill=black] (10.409560784656978,4.256149024714436) circle (1.5pt);
\draw [fill=black] (11.730778486789845,4.291225600877255) circle (1.5pt);
\draw [fill=black] (12.023083288146674,4.2678412167687085) circle (1.5pt);
\draw [fill=black] (13.589837023419276,4.221072448551617) circle (1.5pt);
\draw [fill=black] (13.426146334659451,4.349686561148621) circle (1.5pt);
\draw [fill=black] (12.677846043185971,3.6832316140550527) circle (1.5pt);
\draw [fill=black] (12.572616314697513,3.4844643491324097) circle (1.5pt);
\draw [fill=black] (11.613856566247113,2.5958577530076523) circle (1.5pt);
\draw [fill=black] (11.800931639115483,2.408782680139282) circle (1.5pt);
\draw [fill=black] (11.578779990084293,2.3386295278136435) circle (1.5pt);
\draw [fill=black] (10.52648270519971,1.4967916999059787) circle (1.5pt);
\draw [fill=black] (10.468021744928343,1.2980244349833356) circle (1.5pt);
\draw [fill=black] (12.315388089503502,1.6487901966115293) circle (1.5pt);
\draw [fill=black] (12.619385082914604,1.8241730774256262) circle (1.5pt);
\draw [fill=black] (12.221850553069318,1.8358652694798994) circle (1.5pt);
\draw [fill=black] (13.344300990279539,2.4204748721935556) circle (1.5pt);
\draw [fill=black] (13.504300990279539,2.4604748721935556) circle (1.5pt);
\draw [fill=black] (13.659990175744914,2.4438592563021015) circle (1.5pt);
\draw [fill=black] (13.60152921547355,1.6487901966115293) circle (1.5pt);
\draw [fill=black] (13.636605791636368,1.4149463555260668) circle (1.5pt);
\end{scriptsize}
\end{tikzpicture}}
\caption{An illustration of $n$ points with (a) $n/2$ reflexive pairs for $n=22$, (b) $(n-1)/2$ reflexive pairs for $n=23$, and (c) $(n-3)/2$ reflexive pairs for $n=23$.}
\label{fig:Rcases}
\end{figure}
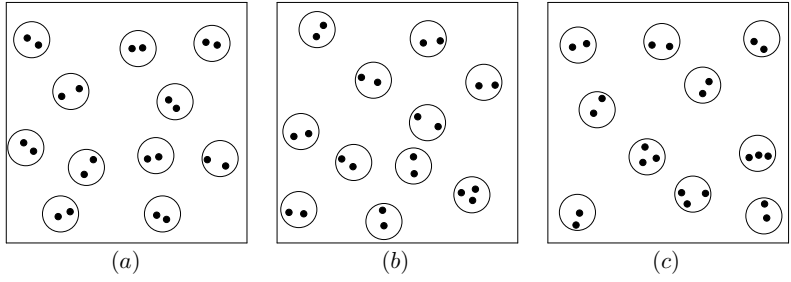

Now suppose that $n$ is an odd positive integer.
The event that each of $(n-3)/2$ sufficiently small balls in the region contains exactly 2 sample points and a sufficiently small ball contains exactly 3 sample points has a positive probability,
and hence $P(R=(n-1)/2)>0$ (See Figure \ref{fig:Rcases} (b)).
Moreover, it is easy to verify that $R=(n-1)/2$ implies
\begin{align*}
Q_0=1, Q_1=n-2,
Q_2=1 \  {\rm and} \ Q_j=0 \  {\rm for\ all}\  3\leq j\leq \kappa' .
\end{align*}
In other words, we have
\begin{align}\label{eq:RdepQj2}
\left \{R=(n-1)/2  \right \} \subseteq \{Q_0=Q_2=1\}\cap \{Q_1=n-2\}
\cap \left( \bigcap_{j=3}^{\kappa' } \{Q_j=0\} \right).
\end{align}
So, by \eqref{eq:RdepQj2} we have $\{R=(n-1)/2\}\cap\{Q_j\geq 1\}=\emptyset$
for all $3\leq j\leq \kappa' $.
Then since $P(Q_j\geq 1)>0$ by Lemma \ref{lem:qjtreshold} and $P(R=(n-1)/2)>0$,
we obtain that $R$ and $Q_j$ are dependent for each $3\leq j\leq \kappa' $.

Furthermore, the event that each of $(n-9)/2$ sufficiently small balls in the region contains exactly 2 sample points and
each of three sufficiently small balls contains exactly 3 sample points has a positive probability (See Figure \ref{fig:Rcases} (c)), and in this case we have
\begin{align}\label{eq:RdepQj3}
R=(n-3)/2, Q_0=3, Q_1=n-6 \  {\rm and}\  Q_2=3.
\end{align}
Therefore, each one of $P(Q_0=3), P(Q_1=n-6)$ and $P(Q_2=3)$ is positive.
By the results in \eqref{eq:RdepQj2} and \eqref{eq:RdepQj3},
we have $\{R=(n-1)/2\}\cap \{Q_0=3\}=\emptyset$,
$\{R=(n-1)/2\}\cap \{Q_1=n-6\}=\emptyset$ and
$\{R=(n-1)/2\}\cap \{Q_2=3\}=\emptyset$.
Hence, we see that $R$ and $Q_j$ are dependent also for every $0\leq j\leq 2$.

Finally, we prove that $R$ and $Q$ are dependent for each $n\geq 6$.
Since $Q=\sum_{j\geq 0} j(j-1)Q_j/2$, by \eqref{eq:RdepQj} we obtain
\begin{align}
\{R=n/2 \}\subseteq \{Q=0\} \label{eq:RdepQ1}
\end{align} when $n$ is even, and
by \eqref{eq:RdepQj2} we have
\begin{align}
\{R=(n-1)/2\}\subseteq \{Q=1\} \label{eq:RdepQ2}
\end{align}
when $n$ is odd.
Clearly, both $R$ and $Q$ always attain nonnegative integer values and $R\leq n/2$ since each point has a unique $NN$.
Hence, \eqref{eq:RdepQ1} and \eqref{eq:RdepQ2} imply
\begin{align}
\{R\geq (n-1)/2\} \subseteq \{Q\leq 1\}.  \label{eq:RdepQ3}
\end{align}
On the other hand, consider the event that each of two sufficiently small balls contains three sample points and remaining $n-6$ sample points are far enough from these two balls.
In this case, the indegrees of the points in each one of the small balls are $0,1$ and 2, and thus, $Q$ is at least 2.
Since such an event occurs with a positive probability, we obtain $P(Q\geq 2)>0$.
Moreover, $P(R\geq (n-1)/2)>0$ and by \eqref{eq:RdepQ3} we have $P(R\geq (n-1)/2, Q\geq 2)=0$.
Therefore, $R$ and $Q$ are dependent as well.
\end{proof}

\begin{remark}
In the proof of Proposition \ref{thm:RdepQj} we use Lemma \ref{lem:qjtreshold} for $k=1$ and
we also need $n$ to be at least 9 to have $(n-9)/2$ balls in the case of odd $n$.
Thus, the lower bound for $n$ is $\max \{9, \kappa' +3\}$.
Note that this lower bound is $\kappa' +3$ for $d\geq 3$, and it is 5 for $d=1$ and 8 for $d=2$.
\end{remark}

We next show that $Q_j$'s are pairwise dependent for large $n$.
\begin{prop}
\label{thm:depQjs}
For every $n \geq 2\kappa' +14$,
$Q_a$ and $Q_b$ are dependent for all $0\leq a\neq b \leq \kappa' $.
\end{prop}
\begin{proof}
We first prove the statement for $1\leq a\neq b \leq \kappa' $.
For each $1\leq j\leq \kappa '(d)$, let $t_j$ and $s_j$ be integers such that
$n=(j+1)t_j+s_j$ and $0\leq s_j \leq j$
(such integers exist by the division algorithm applied to $n$ and $j+1$).
We show that $P(Q_j\geq t_j)>0$.

Consider the balls $B_{\epsilon}(a_0), B_{\epsilon}(a_1), \dots ,B_{\epsilon}(a_j)$ defined in the proof of Lemma \ref{lem:qjtreshold}.
Recall that whenever there exists exactly one point in each of these balls and the remaining sample points are far enough,
the indegree of the point in $B_{\epsilon}(a_0)$ is $j$.
Now consider $t_j$ copies of this configuration and a small ball far from each other.
Suppose each ball of the copies contains exactly one point and the remaining points are in the small ball.
In this case, we have at least $t_j$ points with indegree $j$.
See Figure \ref{fig:Qjtj} for an illustration.
Since having such a configuration is an event with positive probability, we obtain that $P(Q_j\geq t_j)>0$.

\begin{figure}
\centering
\scalebox{.8}{
\begin{tikzpicture}[line cap=round,line join=round,>=triangle 45,x=0.7cm,y=0.7cm]
\clip(0.8065584799999994,0.8379358400000035) rectangle (14.766459160000002,10.333502800000003);
\draw(3.,7.) circle (1.0519521408726079cm);
\draw(11.,8.) circle (1.0480135128662716cm);
\draw(7.,3.) circle (1.052118374984513cm);
\draw(4.5195213333333335,6.999521333333334) circle (0.35cm);
\draw(2.9837227987995663,8.519434225064701) circle (0.35cm);
\draw(1.4811638566971512,6.954370170131628) circle (0.35cm);
\draw(2.96556141262669,5.480868901221287) circle (0.35cm);
\draw(6.985590840411259,4.521115503872058) circle (0.35cm);
\draw(5.48225092689601,2.8978346874326784) circle (0.35cm);
\draw(7.005396585140842,1.4788258229680546) circle (0.35cm);
\draw(8.52,2.94) circle (0.35cm);
\draw(14.,3.) circle (0.35cm);
\draw(11.02,6.52) circle (0.35cm);
\draw(9.484572000000002,7.930426666666668) circle (0.35cm);
\draw(10.985544626852453,9.48006453987222) circle (0.35cm);
\draw(12.480094773396875,7.989070143493331) circle (0.35cm);
\draw [dash pattern=on 2pt off 2pt] (3.,7.) circle (0.35cm);
\draw [dash pattern=on 2pt off 2pt] (11.,8.) circle (0.35cm);
\draw [dash pattern=on 2pt off 2pt] (7.,3.) circle (0.35cm);
\begin{scriptsize}
\draw [fill=black] (2.9839680000000004,8.398938666666668) circle (1.5pt);
\draw [fill=black] (2.86684,6.915317333333335) circle (1.5pt);
\draw [fill=black] (1.3246546666666668,6.8177106666666685) circle (1.5pt);
\draw [fill=black] (3.0230106666666674,5.373132000000002) circle (1.5pt);
\draw [fill=black] (4.291897333333334,6.7396253333333345) circle (1.5pt);
\draw [fill=black] (9.718828000000002,7.7156920000000015) circle (1.5pt);
\draw [fill=black] (11.0658,9.706868000000002) circle (1.5pt);
\draw [fill=black] (11.163406666666669,6.622497333333335) circle (1.5pt);
\draw [fill=black] (12.74463466666667,8.106118666666669) circle (1.5pt);
\draw [fill=black] (11.104842666666668,7.793777333333335) circle (1.5pt);
\draw [fill=black] (7.,3.) circle (1.5pt);
\draw [fill=black] (7.083448000000001,4.631321333333335) circle (1.5pt);
\draw [fill=black] (5.6583906666666675,2.835358666666669) circle (1.5pt);
\draw [fill=black] (6.653978666666668,1.39078) circle (1.5pt);
\draw [fill=black] (8.43042,3.0110506666666685) circle (1.5pt);
\draw [fill=black] (13.779265333333337,2.913444000000002) circle (1.5pt);
\draw [fill=black] (14.247777333333337,3.245306666666669) circle (1.5pt);
\draw [fill=black] (14.15017066666667,2.85488) circle (1.5pt);
\end{scriptsize}
\end{tikzpicture}}
\caption{An illustration for $d=2$, $n=18$, $j=4$, $t_4=3$ and $s_4=3$.
Note that the sample points in the dashed circles are of indegree 4 and hence $Q_4\geq 3$. }
\label{fig:Qjtj}
\end{figure}
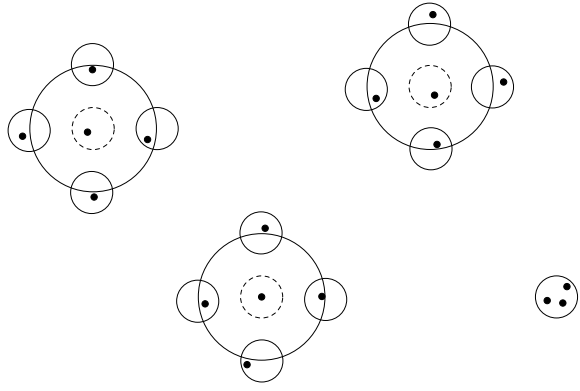

Recall that since the sum of the indegrees in a NND is equal to the number of arcs,
we have $n=\sum_{j=0}^{\kappa' } jQ_j$.
Therefore, whenever both of the events $\{Q_a\geq t_{a}\}$ and $\{Q_b\geq t_{b}\}$ occur,
we have
\begin{align}
n =\sum_{j=0}^{\kappa' } jQ_j\geq a Q_a+b Q_b \geq a t_{a}+bt_{b} >a \left(\frac{n}{a+1}-1 \right)+b \left(\frac{n}{b+1}-1\right) \nonumber
\end{align}
since $t_{a}> \frac{n}{{a}+1}-1$ and $t_b>\frac{n}{b+1}-1$.
Thus, we get
\begin{align}
{a}+b> n\left(\frac{{a}}{{a}+1}+\frac{b}{b+1}-1\right) \nonumber
\end{align}
which is equivalent to
\begin{align}\label{eq:qdep1}
n<\left( {a}+b \right) \frac{{a}b+{a}+b+1}{{a}b-1}=({a}+b) \left(1+\frac{{a}+b+2}{{a}b-1}\right).
\end{align}
We may assume ${a}>b$.
Then ${a}\geq 2$ and ${a}b-1\geq b({a}-1)$, thus we have
\begin{align}
&({a}+b) \left(1+\frac{{a}+b+2}{{a}b-1}\right)\leq ({a}+b) \left(1+\frac{{a}+b+2}{b({a}-1)}\right) \nonumber\\
&={a}+b+\frac{{a}^2+b^2+2{a}b+2{a}+2b}{b({a}-1)} \nonumber\\
&={a}+b+\frac{{a}({a}-1)+b^2+2({a}-1)b+3({a}-1)+4b+3}{b({a}-1)} \nonumber\\
&={a}+ \left(b+\frac{{a}}{b} \right)+\frac{b}{{a}-1}+2+\frac{3}{b}+\frac{4}{{a}-1}+\frac{3}{b({a}-1)} \nonumber\\
&\leq {a}+({a}+1)+1+2+3+4+3=2{a}+14 \leq 2\kappa' +14 \nonumber
\end{align}
since $b+\frac{{a}}{b}\leq {a}+1$ and $\frac{b}{{a}-1}\leq 1$.
Combining this result with the one in \eqref{eq:qdep1} gives $n< 2\kappa' +14$, which is a contradiction.
Thus, we have $P(\{Q_a\geq t_{a}\}\cap \{Q_b\geq t_b\})=0$ and
see that $Q_a$ and $Q_b$ are dependent whenever $1\leq {a}\neq b\leq \kappa '(d)$.

Now recall that if $n$ is even $\{R=n/2 \}\subseteq \{Q_0=0\}$ by \eqref{eq:RdepQj} and $P(R=n/2)>0$,
and if $n$ is odd $\{R=(n-1)/2\}\subseteq \{Q_0=1\}$ by \eqref{eq:RdepQj2} and $P(R=(n-1)/2)>0$.
Thus, we conclude that $P(Q_0\leq 1)>0$.
In addition, as
\begin{align*}
\sum_{i=0}^{\kappa' }  Q_i=n=\sum_{i=0}^{\kappa' }  iQ_i,
\end{align*}
we get
\begin{align*}
Q_0=\sum_{i=2}^{\kappa' }  (i-1)Q_i=Q_2+2Q_3+3Q_4+\cdots + (\kappa' -1)Q_{\kappa' }.
\end{align*}
Therefore, if $Q_0\leq 1$, then $Q_j=0$ for all $3\leq j \leq \kappa' $, $Q_2=Q_0\leq 1$, and $Q_1=n-2Q_0\geq n-2$,
i.e.,
\begin{align}\label{eq:qdep2}
\{Q_0\leq 1\}\subseteq \{Q_1\geq n-2\} \cap \{Q_2\leq 1\} \cap \left( \bigcap_{j=3}^{\kappa' } \{Q_j=0\} \right).
\end{align}
Whenever $3\leq j \leq \kappa' $, by Lemma \ref{lem:qjtreshold} we have $P(Q_j\geq 1)>0$, and also,
by \eqref{eq:qdep2} we obtain $P(Q_0\leq 1, Q_j\geq 1)=0$.
Consequently, $Q_0$ and $Q_j$ are dependent for every $3\leq j \leq \kappa' $.

Since $n\geq 2\kappa' +14\geq 18$, we get $t_2\geq 6$,
and therefore $P(Q_2\geq 6)>0$.
Then \eqref{eq:qdep2} implies $\{Q_0\leq 1\}\cap \{Q_2\geq 6\}=\emptyset$, and so, $Q_0$ and $Q_1$ are dependent.
As $\sum_{i=0}^{\kappa' } Q_i=n$ and $P(Q_2\geq 6)>0$, we have $P(Q_1\leq n-6)>0$.
Then, by \eqref{eq:qdep2} we see that $\{Q_0\leq 1\}\cap \{Q_1\leq n-6\}=\emptyset$
and obtain the dependence of $Q_0$ and $Q_1$ as well.
\end{proof}

\begin{remark}
Recall that the sample size is not fixed in $\mathcal{P}_n$,
and hence we can not apply the arguments used in this section for quantities based on $\mathcal{P}_n$.
\end{remark}

\section{Discussion and Conclusions}
\label{sec:tr2discconc}
In this paper, we study the asymptotic behavior of the number of copies of minuscule constructs in $k$NN digraphs of random point sets.
As point processes, we consider the uniform binomial point process and HPP over a given region.
For any realization of the point set, consider the $k$NN digraph of the data.
The quantity we are interested in is the number of subdigraphs of the $k$NN digraph which are isomorphic to a given weakly connected digraph.
We provide LLN and CLT results for any linear combination of such quantities.
In particular, we focus on the number reflexive pairs, the number of shared $k$NN's and the number of vertices with a given indegree.
A potential research direction is to consider the same $k$NN invariants under different point processes.
Monte Carlo simulations suggest the asymptotic normality of the quantities we study whenever the underlying process is a distribution with an a.e. continuous density.

Notice that the condition on the minuscule construct being weakly connected is crucial.
Because, if the fixed digraph is not weakly connected, then the strong stabilization condition fails in the proof of our main theorem.

All the asymptotic results we present have analogous versions for graphs and marked point sets as stated in
Remarks \ref{rem:kNNgraphversion} and \ref{rem:markedppversion}.
However, we prefer to mainly study on digraphs as the random variables we are interested in are based on $k$NN digraphs.

\end{document}